\documentclass[11pt,twoside]{article}
\usepackage{amsmath,amssymb,amsthm,epsfig,subfigure}
\newtheorem{proposition}{Proposition}
\newtheorem{theorem}[proposition]{Theorem}
\newtheorem{lemma}[proposition]{Lemma}
\newtheorem{corollary}[proposition]{Corollary}

\newtheorem{remark}[proposition]{Remark}

\newtheorem{algorithm}{Algorithm}
\newtheorem*{algorithm*}{Algorithm}
\newtheorem{assumption}{Assumption}

\newcommand{\etab}{{\boldsymbol{\eta}}}
\newcommand{\eb}{{\bf e}}
\newcommand{\gb}{{\bf g}}
\newcommand{\hb}{{\bf h}}
\newcommand{\fb}{{\bf f}}

\newcommand{\xb}{{\bf x}}
\newcommand{\yb}{{\bf y}}
\newcommand{\zb}{{\bf z}}
\newcommand{\rb}{{\bf r}}
\newcommand{\tb}{{\bf t}}
\newcommand{\vb}{{\bf v}}
\newcommand{\ub}{{\bf u}}
\newcommand{\pb}{{\bf p}}

\newcommand{\R}{\mathbb{R}}

\newcommand{\N}{\mathbb{N}}

\newcommand{\diff}{{\rm d}}

\newcommand{\rank}{{\rm rank}}

\newcommand{\BE}{\begin{equation}}
\newcommand{\EE}{\end{equation}}

\newcommand{\oneb}{\boldsymbol{1}}
\newcommand{\zerob}{\boldsymbol{0}}

\renewcommand{\S}{\mathbf{S}}

\begin{document}
\title{Regularization preconditioners for frame-based \\ image deblurring with reduced boundary artifacts}
\author{
Yuantao Cai\thanks{School of Mathematical Sciences, University of Electronic Science and Technology of China, Chengdu, Sichuan, 611731, PR China. 
E-mail: yuantaocai@163.com}, 
        Marco Donatelli\thanks{Dipartimento di Scienza e Alta Tecnologia, Universit\`a dell'Insubria, 22100 Como, Italy,
        E-mail: marco.donatelli@uninsubria.it},
Davide Bianchi\thanks{Dipartimento di Scienza e Alta Tecnologia, Universit\`a dell'Insubria, 22100 Como, Italy,
        E-mail: d.bianchi9@uninsubria.it},
Ting-Zhu Huang\thanks{School of Mathematical Sciences, University of Electronic Science and Technology of China, Chengdu, Sichuan, 611731, PR China. 
E-mail: tingzhuhuang@126.com}}
\maketitle

\abstract{
Thresholding iterative methods are recently successfully applied to image deblurring problems.
In this paper, we investigate the modified linearized Bregman algorithm (MLBA) used in  image deblurring problems, with a proper treatment of the boundary artifacts. We consider two standard approaches: the imposition of boundary conditions and the use of the rectangular blurring matrix.

The fast convergence of the MLBA depends on a regularizing preconditioner that could be computationally expensive and hence it is usually chosen as a block circulant circulant block (BCCB) matrix, diagonalized by discrete Fourier transform.
We show that the standard approach based on the BCCB preconditioner may provide low quality restored images and we propose different preconditioning strategies, that improve the quality of the restoration and save some computational cost at the same time.
Motivated by a recent nonstationary preconditioned iteration, we propose a new algorithm that combines such method with the MLBA. We prove that it is a regularizing and convergent method. A variant with a stationary preconditioner is also considered. 
Finally, a large number of numerical experiments shows that our methods provide accurate and fast restorations, when compared with the state of the art.
}

\section{Introduction}
Image deblurring is the process of reconstructing an approximation of an image from blurred
and noisy measurements.
By assuming that the point spread function (PSF) is spatially-invariant,
the observed image $g(x,y)$ is related to the true image $f(x,y)$
via the integral equation
\begin{equation} \label{eq:model2}
g(x,y) = \int\limits_{-\infty}^{+\infty} \int\limits_{-\infty}^{+\infty} h(x-
x',y-y') f(x',y')~dx'~dy' + \eta(x,y),  \qquad(x,y) \in \Omega \subset\R^2,
\end{equation}
where  $\eta(x,y)$ is the noise.

By collocation of the previous integral equation on a uniform grid, we obtain the grayscale images of the observed image, of the true image, and of the PSF, denoted by $G$, $F$, and $H$, respectively.
Since collected images are available only in a finite region, the field of view (FOV),
the measured intensities near the boundary are affected by
data outside the FOV.
Given an $n \times n$ observed image $G$ (for the sake of simplicity we assume square images), and a $p \times p$ PSF with $p\leq n$, then $F$ is $m \times m$ with $m=n+p-1$.
Denoting by $\gb$ and $\fb$ the stack ordered vectors corresponding to $G$ and $F$, the discretization of \eqref{eq:model2} leads to the under-determined linear system
\begin{equation} \label{eq:modeld}
\gb = A \fb + \etab,
\end{equation}
where the matrix $A$ is of size $n^2 \times m^2$. When Imposing proper Boundary Conditions (BCs), the image $A$ becomes square $n^2 \times n^2$ and in some cases, depending on the BCs and the symmetry of the PSF, it can be diagonalized by discrete trigonometric transforms. For example, the matrix $A$ is block circulant circulant block (BCCB) and it is diagonalizzable by Discrete Fourier Transform (DFT), when periodic BCs are imposed. 

Due to the ill-posedness of \eqref{eq:model2}, $A$ is severely ill-conditioned and may be singular. In such case, linear systems of equations \eqref{eq:modeld} are commonly referred to as linear discrete ill-posed problems \cite{HH}.
Therefore a good approximation of $\fb$ cannot be obtained from the algebraic solution (e.g., the least-square solution) of
\eqref{eq:modeld}, but regularization methods are required. The basic idea of regularization is to replace the original ill-conditioned
problem with a nearby well-conditioned problem, whose solution approximates
the true solution. One of the popular regularization techniques is the Tikhonov regularization and it amounts in solving
\begin{equation}\label{eq:tik}
\min_{\fb}\{\|A\fb - \gb\|_2^2 + \mu \|\fb\|_2^2\},
\end{equation}
where $\|\cdot \|_p$ denotes the vector $p$-norm, $p\geq1$, and $\mu>0$ is a regularization parameter to be chosen.
Hereafter, we use $\|\cdot\| \equiv \|\cdot\|_2$ to denote the $\ell_2$-norm.
The first term in \eqref{eq:tik} is usually refereed as fidelity term and the second as regularization term.
This approach is computationally attractive, since it leads to a linear problem and indeed several efficient methods have been developed for computing its solution and for estimating $\mu$ \cite{HH}.
On the other hand, the edges of restored image are usually over-smoothed.
To overcome this unpleasant property, nonlinear strategies have been employed, like total variation (TV) \cite{ROF} and
thresholding iterative methods \cite{D3,FN03}. Anyway, several nonlinear regularization methods have an inner step that apply a least-square regularization and hence can benefit from strategies previously developed for such simpler model, as we will show in the following.

In this paper we consider a regularization strategy based on wavelet decomposition that has been recently largely investigated \cite{COS,COS2,CRSS, FN03, DHRZ, D3}.
This approach is motivated by the fact that most real images usually have sparse approximations under
some wavelet basis. In particular, in this paper we consider the tight frame systems previously used in \cite{CCSS03,COS,COS2}.
Solving \eqref{eq:modeld} in a tight frame domain, the
redundancy of system leads to robust signal representation in which partial loss of
the data can be tolerated without adverse effects. In order to obtain the sparse approximation, we minimize the weighted $\ell_1$-norm of the tight frame coefficients.
Let $W^T$ be a wavelet or tight-frame synthesis operator ($W^TW=I$), the
wavelets or tight-frame coefficients of the original image $\fb$ are $\xb$ such that
\begin{equation}\label{eq:wavecoeff}
\fb= W^T\xb.
\end{equation}
In the following, we will investigate the synthesis approach, but our proposal can be applied also to the analysis and to the balanced approach described in \cite{COS2,STY}.
Reformulating the deblurring problem \eqref{eq:modeld} in terms of frame coefficients
\begin{equation} \label{eq:l1}
\min_{\xb} \{\mu\|\xb\|_1 + \|\xb\|^2: AW^T\xb =\gb\},
\end{equation}
a regularized solution of this problem can be obtained by the Bregman splitting
\cite{YOGD}. As for the iterative soft-thresholding \cite{D3,FN03} for the unconstrained version of \eqref{eq:l1} and the Landweber method for the least-square solution of \eqref{eq:modeld},  the Bregman splitting converges very slowly
for image deblurring problems. Hence a preconditioning strategy is usually employed, obtaining the Modified Linearized Bregman Algorithm (MLBA) \cite{COS}. The preconditioner is usually chosen as a BCCB approximation of $(AA^T+\alpha I)^{-1}$, $\alpha>0$, see \cite{COS,COS2,STY}, which is the simplest regularized version of the inverse of $AA^T$ which, also when theoretically available, cannot be computed due to the
severe ill-conditioning of $A$.

In this paper, we show that the BCCB preconditioner used in the literature leads often to poor restorations when the matrix $A$ has the rectangular
$n^2 \times m^2$ structure or is obtained by imposing accurate BCs, like antireflective BCs \cite{Serra}. Note that in real applications we have to take into account the boundary effects to obtain high quality restorations, otherwise the restored image is severely affected by ringing effects \cite{HNO05,NCT99}. This topic has been recently investigated in connection with nonlinear models based on wavelets or TV  in \cite{S12,AF13,BCDS14}, but, at our knowledge, this is the first time that it is considered in connection with the MLBA.
In this context we propose and discuss other preconditioning strategies for the MLBA with the synthesis approach. Our preconditioners are inspired by the experience with least-square regularization where the regularization preconditioning is studied since a long time, see the seminal paper \cite{HNP}. 
In particular a nonstationary preconditioned iteration suggested by \cite{DH13} leads to a new algorithm that is no longer a Bregman iteration.

We investigate the following two strategies to define accurate and computationally cheap preconditioners:
\begin{enumerate}
	\item[(1)] an approximation of the blurring operator in a small Krylov subspace;
	\item[(2)] a symmetrization of the original PSF $H$;
\end{enumerate}
The choice (1) is quite natural and already considered for similar problems (see e.g. \cite{AF13}), but we will show that a properly chosen Krylov subspace of small size (say spanned by at most five vectors), with a proper choice of the initial guess, is usually enough to obtain a good approximation. The choice (2) can be very useful in many applications where the PSF is obtained experimentally by measurements and is a perturbation of a symmetric kernel. In this case the approximated quadrantally symmetric (i.e., symmetric with respect to each quadrant) PSF leads to a matrix diagonalizable by Discrete Cosine Transform (DCT). 

Following the idea to combine preconditioned regularizing iterative methods for least-square ill-posed problems with the MLBA, we propose a new algorithm based on the recent proposal in \cite{Acqua,DH13}. The nonstationary preconditioner is defined by a parameter computed by solving a nonlinear problem with a computational cost of $O(n^2)$.
We observe that this method can be applied only with square matrices and so only when $A$ is obtained by imposing BCs.
The new algorithm is no longer a Bregman iteration and we cannot apply the convergence analysis developed in \cite{COS}
for the MLBA. Therefore, here we prove its convergence and its regularization character.
Furthermore, when a good value for the parameter in the preconditioner is available, we provide a variant of the algorithm with a stationary preconditioner which can improve the quality of the restorations even if the previous convergence analysis does not hold any longer.

A large number of numerical experiments in Section~\ref{sec:numres} shows that our proposals 
not only outperform the standard MLBA with BCCB preconditioning, but are also good competitors for other recent methods dealing with boundary artifacts proposed in \cite{AF13,BCDS14}.

Besides, we mention that the two deblurring models based on the rectangular matrix $A$ and the imposition of BCs, we have tested also a third strategy based on the enlargement of the domain to reduce the ringing effects like in \cite{R05,S12}, but, according to the results in \cite{AF13}, the quality of the restored images were not better than those obtained with the other two models, while the CPU time was higher. Hence, we do not discuss further this strategy \nolinebreak here.

The paper is organized as follows. In Section~\ref{sect:structA}, we describe  the structure of the blurring matrix $A$ explaining how fast trigonometric transforms can be used in the computations both for the rectangular matrix and the square matrices arising from the imposition of classical BCs.
Section~\ref{sect:mlb} reviews briefly the synthesis approach and the MLBA for solving the corresponding minimization problem. In Section~\ref{sect:prec} we propose possible regularization preconditioners, combining accurate restorations and a low computational cost.   
A new algorithm is proposed in Section \ref{sect:alg4} combining the MLBA with the method in \cite{DH13}.
Section~\ref{sec:numres} contains a large number of numerical experiments, comparing our proposal with some state of the art algorithms, for the restoration of images with unknown boundaries.
Concluding remarks are provided in Section~\ref{sect:concl}.

\section{The structure of the blurring matrix}\label{sect:structA}
We count mainly three strategies in  order to obtain both accurate and fast restorations
with reduced boundary artifacts. In this paper, we just consider two of them: the use of the original rectangular matrix and the imposition of BCs.
As mentioned in the Introduction, we do not consider the third strategy introduced in \cite{R05}, since from one side it is equivalent to the reflective (or Neumann) BCs in the case of quadrantally PSF, cf. \cite{DS10}, and from the other side, according to several tests that we have performed and the numerical results in \cite{AF13}, it does not provide a better restoration than the other two strategies, while  it usually requires a larger CPU time.

In this section we describe the structure of the matrix $A$ and how fast computations with such matrix, like matrix-vector product or least-square solutions, can be implemented. Firstly we introduce the rectangular $n^2 \times m^2$ matrix and then the square $n^2 \times n^2$ matrix obtained when imposing proper BCs.

\subsection{The rectangular matrix}\label{ssect:rect}

The fact that this matrix is not square prevents the use of Fast Fourier Transform (FFT).
To cope with this difficulty, one can construct an $m^2\times
m^2$ blurring matrix $A_{big}$ that is BCCB, and hence, the FFT can be used.
Let $M\in\R^{n^2\times m^2}$ be a masking matrix which, when applied to a vector
in $\R^{m^2}$, selects only the entries in the FOV, i.e., their rows are a subset of
the rows of an identity matrix of order $m^2$.
The rectangular blurring matrix can be written as
\begin{equation}\label{eq:Arect}
 A = MA_{\rm big}.
 \end{equation}
Hence the matrix vector $A\xb$ can be easily computed by two bi-dimensional FFTs of
order $m^2$, followed by a selection of the pixels inside the FOV.
Similarly  $A^T\xb = A^T_{\rm big}M^T\xb$ and thus two FFTs are applied to the zero-padded version of
$\xb$ of size $m^2$.
This approach was used in \cite{AF13} and it is numerically equivalent to that adopted in \cite{VBDW05}.

We observe that the matrix $M$ in \eqref{eq:Arect} leads to a matrix $A$ independent of the BCs 
used in the definition of $A_{\rm big}$. Thus, when the PSF is quadrantally symmetric, we suggest 
to use of the DCT instead of the DFT
(see the discussion on reflective BCs in the next subsection).

The structure of the matrix $A$ and its representation in  \eqref{eq:Arect} allow fast computations of
the matrix vector product with $A$ and $A^T$, but they prevent the use of fast transforms for solving
linear systems with polynomials of $AA^T$ as coefficient matrix even when $A$ is full-rank.

\subsection{Boundary conditions}\label{ssect:BC}
The BC approach forces a functional dependency between the
elements of $\fb$ external to the FOV and those internal to this
area. This has the effect of extending $F$ outside of the FOV
without adding any unknowns to the associated image deblurring
problem. Therefore, the matrix $A$ can be written as an $n^2 \times n^2$ square matrix,
whose structure can be exploited by fast algorithms.
If the BC model is not a good approximation of the real world outside the FOV,
the reconstructed image can be severely affected by some unwanted artifacts near the
boundary, called ringing effects \cite{HNO05}.

The use of different BCs can be motivated from information on the true image
and/or from the availability of fast transforms to diagonalize the matrix $A$ within $O(n^2\log(n))$ arithmetic operations.
Indeed, the matrix-vector product can be always computed by the 2D FFT, after a proper padding of the image to convolve (the resulting image is the inner $n \times n$ part of the convolution), cf. \cite{RestoreTools}, while the availability of fast transforms to diagonalize the matrix $A$ depends on the BCs. Anyway, the shift-invariant property of the blur leads to a matrix $A$ that can be well approximated by a BCCB matrix $C$, which is diagonalized by DFT, because usually in the applications 
\begin{equation}\label{eq:apprxAP}
A-C=R+N,
\end{equation}
where $R$ is a matrix of small rank and $N$ is a matrix of small norm. 
More precisely, for any $\varepsilon>0$ there is a constant $c_\epsilon >0$ independent of $n$ and depending only on $\epsilon$ and on the PSF, such that the splitting \eqref{eq:apprxAP} holds with 
\begin{equation}\label{eq:ranknorm}
\rank(R)\leq c_\epsilon \cdot n, \qquad \|N\|\leq\varepsilon,
\end{equation}
where $\rank(R)$ denotes the rank of $R$ (see \cite{HN96}).
Note that $n^2$ is the size of the matrix $A$.

In the following we recall common BCs that will be used in the numerical tests.
For a detailed description of zero, periodic, and reflective, refer to \cite{HNO05},
while for antireflective BC see the review paper \cite{DS10rev} and the original proposal in \cite{Serra}.

\begin{table}
\begin{center}\begin{small}
\begin{tabular} {c@{\hspace{1.5cm}}c@{\hspace{1.5cm}}c} 
Zero & Periodic & Reflective \\
\hline\\
$\begin{array}{ccc}
\zerob & \zerob & \zerob \\
\zerob & F      & \zerob \\
\zerob & \zerob & \zerob
\end{array} $
& 
$\begin{array}{ccc}
F & F & F \\
F & F & F \\
F & F & F
\end{array} $ 
& 
$\begin{array}{ccc}
F_{rc} & F_r & F_{rc} \\
F_c    & F   & F_c \\
F_{rc} & F_r & F_{rc}
\end{array} $
\end{tabular}
\end{small}
\caption{Pad of the original image $F$ obtained by imposing the classical BCs considered in \cite{HNO05}, with
$F_c = {\rm fliplr}(F)$,
$F_r = {\rm flipud}(F)$, and
$F_{rc}= {\rm flipud}({\rm fliplr}(F))$,
where ${\rm fliplr}(\cdot)$ and ${\rm flipud}(\cdot)$ are the MATLAB functions that perform the left-right and up-down flip, respectively.
}
\label{tab:BC}
\end{center} 
\end{table}

\begin{description}
\item[Zero] (i.e., Dirichlet) BCs
assume that the object is zero outside of the FOV.
That is, one assumes that $F$ has been extracted from a larger array 
padded by zeros (see Table\ref{tab:BC}).
This is a good choice when the true image is mostly zero
outside the FOV, as is the case for many astronomical or medical images with a black background.
Unfortunately, these BCs have a bad effect on reconstructions of images
that are nonzero outside the border, leading to reconstructed image with severe ``ringing'' near the boundary.
The corresponding matrix $A$ has a block-Toeplitz-Toeplitz-block (BTTB) structure which is not diagonalizable by fast trigonometric transforms.
\item[Periodic] BCs assume
that the observed image repeats in all directions.
More specifically, one assumes that $F$ has been extracted from a larger array of the form in Table~\ref{tab:BC}.
The corresponding matrix $A$ is BCCB and so is always diagonalized by 2D DFT. Clearly, if the true image is not periodic outside the FOV the reconstructed image will be affected by severe ringing effects.
\item[Reflective] (i.e., Neumann or symmetric) BCs
assume that outside the FOV the image is a mirror image of $F$ \cite{NCT99}.
That is, one assumes that $F$ has been extracted from a larger
array symmetrically padded like in Table~\ref{tab:BC}.
The matrix $A$ has a block structure that combines Toeplitz and Hankel structures,
but it can be easily diagonalized by the DCT,
when the PSF is quadrantally symmetric \cite{NCT99}.
\item[AntiReflective] BCs
have a more elaborate definition, but have a simple motivation:
the antisymmetric pad yields an extension that preserves the continuity of the normal derivative \cite{Serra}.
They are given by
\begin{align*}
F(1-i,j)  = 2 F(1,j) & - F(i+1,j), & 1 \leq i \leq p, \;1 \leq j \leq n; \\
F(i,1-j)  = 2 F(i,1) & - F(i,j+1) , & 1 \leq i \leq n, \;1 \leq j \leq p;  \\
F(n+i,j)  = 2 F(n,j) & - F(n-i,j), & 1 \leq i \leq p, \;1 \leq j \leq n; \\
F(i,n+j)  = 2 F(i,n) & - F(i,n-j), & 1 \leq i \leq n, \;1 \leq j \leq p;
\end{align*}
for the edges, while for the corners the more computationally
attractive choice is to antireflect first in one direction and
then in the other \cite{spie2003}. This yields
\begin{align*}
F(1-i,1-j) = & \; 4 F(1,1) - 2F(1,j+1) - 2F(i+1,1) + F(i+1,j+1), \nonumber \\
F(1-i,n+j) = & \; 4 F(1,n) - 2F(1,n-j) -2F(i+1,n) + F(i+1,n-j), \nonumber \\
F(n+i,1-j) = & \; 4 F(n,1) - 2F(n,j+1) -2F(n-i,1) + F(n-i,j+1), \nonumber \\
F(n+i,n+j) = & \; 4 F(n,n) - 2F(n,n-j) -2F(n-i,n) + F(n-i,n-j), \nonumber
\end{align*}
for $1 \leq i, j \leq p$.  

The structure of the matrix $A$ is quite involved, but it can be diagonalized by the antireflective transform, 
when the PSF is quadrantally symmetric \cite{ADS08}. Since $A$ is not normal the antireflective transform is not unitary, but it can be represented as a modification of the discrete sine transform, formed by adding a uniform sampling of constant and linear functions to the eigenvector basis preserving an ``almost'' unitary behaviour~\cite{DS10rev}.\footnote{MATLAB functions for working with antireflective BC (antireflective transform, antisymmetric pad, ecc.) can be download at {\sf http://scienze-como.uninsubria.it/mdonatelli/Software/software.html}}

Due to the structure of $A$, the application of $A^T$ could generate artifacts at the boundary \cite{DS}; consequently it was proposed to replace $A^T$ by a reblurring matrix $A'$ obtained by imposing the same BCs to the PSF rotated by $180^\circ$ \cite{DEMS06}.
Note that in the case of zero and periodic BCs $A'=A^T$.
Furthermore, the MATLAB Toolbox RestoreTools \cite{NPP} (that we will use in the numerical results) implements the reblurring approach to overload the matrix-vector product with $A^T$, see \cite{DEMS06}. Therefore, for the sake of notational simplicity and uniformity, in the following the symbol $A^T$ has to be intended as the reblurring matrix $A'$ in the case of antireflective BCs.
\end{description}

The reflective BCs will not be considered in the numerical results in Section \ref{sec:numres} since they have the same computational properties of the antireflective BCs, e.g., for quadrantally PSF simply replace the antireflective transform with the DCT, and usually provide restorations slightly worser or at least comparable with the antireflective BCs, as we have numerically observed according to similar results with other regularization strategies \cite{DEMS06,P06,CH08}.
On the other hand, we have recalled also the reflective BCs to motivate a preconditioner based on the DCT, instead of the DFT, when the PSF is quadrantally symmetric.

\section{MLBA for the synthesis approach}\label{sect:mlb}
For the synthesis approach \cite{COS, FN03} the coefficient matrix is
\begin{equation}\label{eq:K}
K =A W^T \in \R^{n^2\times s},
\end{equation}
where $n^2\leq s$, $A$ is the blurring matrix and $W^T$ is a tight-frame or wavelet synthesis operator.
Note that using tight-frames $W^TW=I$ but $WW^T\neq I$ \cite{DHRZ}.
The use of tight-frames instead of wavelets is motivated by the fact that
the redundancy of tight-frame systems leads to robust signal representations in which partial loss of the data can be tolerated, without adverse effects, see e.g. \cite{CRSS}.

Denote by $\xb$ the frame coefficients of the image $\fb$ according to \eqref{eq:wavecoeff}.
Let the nonlinear operator $\S_\mu$ be defined component-wise as
\begin{equation}\label{eq:soft}
[\S_\mu(\xb)]_i=S_\mu(x_i),
\end{equation}
with $S_\mu$ the soft-thresholding function
\[
S_\mu(x_i)={\rm sgn}(x_i)\max\left\{|x_i|-\mu,\,0\right\}.
\]

Note that for image deblurring problems the singular values of $A$, and so those of $K$, decay exponentially to zero and we cannot assume that $K$ is surjective.
Therefore, the deblurring problem can be reformulated in terms of the frame coefficients $\xb$ as
\begin{equation} \label{eq:l1bis}
\min_{\xb\in\R^s} \left\{\mu\|\xb\|_1 + \frac{1}{2\lambda}\|\xb\|^2 : {\rm arg} \min_{\xb\in\R^s} \|K\xb -\gb\|^2\right\}.
\end{equation}
which is equivalent to \eqref{eq:l1} if $K$ is surjective.

The following linearized Bregman iteration
\begin{equation}\label{eq:iterLinBreg}
\left\{ {{\begin{array}{ll}
 {\zb^{n+1} =\zb^{n}  + K^T\left(\gb-K\xb^{n}\right),}\hfill \\
  {\xb^{n+1} = \lambda\S_\mu(\zb^{n+1}),} \hfill \\
\end{array} }} \right.
\end{equation}
where $\zb^0 = \xb^0 = \bf{0}$, was introduced in  \cite{YOGD} to solve problem \eqref{eq:l1bis} and later applied to image deblurring in \cite{COS}. A detailed convergence analysis of the linearized Bregman iteration \eqref{eq:iterLinBreg} was given in \cite{COS3} when $K$ is surjective, but we report here Theorem 3.1 in \cite{COS} that does not require such assumption.
\begin{theorem}[\cite{COS}] \label{lem:LinBregConv}
Let $K\in\R^{n^2\times s}$, $n^2<s$ and let
$0<\lambda< \frac{1}{\|K^TK\|}$. Then the sequence $\{\xb^{n+1}\}$
generated by \eqref{eq:iterLinBreg}
converges to the unique solution of
\eqref{eq:l1bis}.
\end{theorem}

As observed in \cite{COS3} the convergence speed of \eqref{eq:iterLinBreg} depends of the condition number of $K$ which, as observed before, is very large for image deblurring and hence the method results to be very slow.
To accelerate its convergence in the case of $KK^T\neq I$, in \cite{COS} the authors modified iteration \eqref{eq:iterLinBreg} by replacing $K^T$ with $K^\dag$, where $K^\dag$ denotes the pseudo-inverse of $K$.
If $K$ is surjective $K^\dag=K^T(KK^T)^{-1}$ since $n^2\leq s$. For image deblurring problems they suggested to replace $(KK^T)^\dag$ with a symmetric positive definite matrix $P$ such that
\begin{equation}\label{eq:P}
P \approx (KK^T)^\dag=(AA^T)^{\dag}.
\end{equation}
Then the MLBA for frame-based image deblurring becomes \cite{COS}
\begin{equation} \label{eq:MLBA}
    \left\{
    \begin{array}{l}
          \zb^{n+1}=\zb^{n}+K^TP( \gb- K\xb^{n}), \\
          \xb^{n+1}=\lambda\S_\mu(\zb^{n+1}),
    \end{array}
    \right.
\end{equation}
where $\zb^0 = \xb^0 = \bf{0}$.

\begin{remark}\label{rem:prec}
The MLBA \eqref{eq:MLBA} is the linearized Bregman iteration \eqref{eq:iterLinBreg} for the 
linear system
\[
P^{1/2}K\xb=P^{1/2}\gb,
\]
which is a preconditioned version of original linear system $K\xb=\gb$ by the preconditioner $P^{1/2}$.
In fact, by replacing $K$ and $\gb$ in \eqref{eq:iterLinBreg} by $P^{1/2}K$ and $P^{1/2}\gb$, respectively,
we obtain~\eqref{eq:MLBA}.
\end{remark}

The previous remark is the key observation used in \cite{COS} to prove that the MLBA algorithm converges to a minimizer of
\begin{equation} \label{eq:objLinBregLS}
\min_{\xb\in\R^s}\left\{\mu \|\xb\|_1 + \frac{1}{2\lambda}\|\xb\|^2: \xb =
 {\rm arg} \min_{\xb\in\R^s} \|K\xb -\gb\|_P^2\right\}.
\end{equation}

\begin{theorem}[\cite{COS}] \label{lem:modLinBregConv}
Assume $P$ is a symmetric positive definite matrix and let
$0<\lambda< \frac{1}{\|K^TPK\|}$. Then the sequence $\{\xb^{n+1}\}$
generated by the MLBA~\eqref{eq:MLBA}
converges to the unique solution of
\eqref{eq:objLinBregLS}.
\end{theorem}

The standard choice for $P$ is
\begin{equation}\label{eq:Pstandard}
P=(KK^T+\alpha I)^{-1}=(AA^T+\alpha I)^{-1}.
\end{equation}
In such case
\[\|K^TPK\|<1\]
and hence $\lambda=1$ is a good choice according to Theorem \ref{lem:modLinBregConv}.
When $AA^T+\alpha I$ is not easily invertible other choices as $P$ can be explored, but usually the quantity $\|K^TPK\|$ becomes hard to estimate. Therefore, assuming that $P$ is a good approximation of $(AA^T+\alpha I)^{-1}$, we set $\lambda=1$ in the algorithm. Anyway, the validity of the assumption $\|K^TPK\|<1$ can be guaranteed choosing $\alpha$ large enough. 

In conclusion, we consider the following version of the MLBA for the synthesis approach
\begin{equation} \label{eq:MLBA2}
    \left\{
    \begin{array}{l}
          \zb^{n+1}=\zb^{n}+WA^TP( \gb- AW^T\xb^{n}), \\
          \xb^{n+1}=\S_\mu(\zb^{n+1}),
    \end{array}
    \right.
\end{equation}
stopped by the \emph{discrepancy principle} as in \cite{COS}, i.e., at the first iteration $n=\tilde n>0$ such that
\begin{equation}\label{eq:discrpr} 
\|\rb^{\tilde n}\|\leq \gamma \delta < \|\rb^{n}\|, \qquad n=0,1,\dots, \tilde{n}-1,
\end{equation}
where $\gamma>1$, $\delta=\|\etab\|$, and $\rb^{n}=\gb- AW^T\xb^{n}$ is the residual at the $n$-th iteration.
Here $\zb^0 = \xb^0 = \bf{0}$ and we assume that the noise level $\delta$ is explicitly known.


\section{On the choice of the preconditioner $P$}\label{sect:prec}

In this section we explore possible choices of $P\neq(AA^T+\alpha I)^{-1}$ which are computationally attractive.
Let $A$ be the rectangular, anti-reflective or BTTB matrix depending on the chosen treatment of the boundary of the image,
and let $C$ be the BCCB obtained from the same PSF.
Since the matrix $P$ in Theorem~\ref{lem:modLinBregConv} serves as a preconditioner to accelerate the convergence,
in this section we describe some preconditioning strategies, in order to combine fast computations with accurate restorations achievable
when $P \approx (AA^T+\alpha I)^{-1}$.
The first proposal in Section~\ref{sect:BCCBprec} is the classical approach already used in the literature, cf. \cite{COS, COS2, STY}.
The second proposal in Section~\ref{sect:Krylovprec} is an approximation strategy considered for similar methods, c.f. \cite{AF13}, but this is the first time that it is explored with MLBA. The third proposal is inspired by a similar approach used with numerical methods that require symmetric matrices, see \cite{HN96,NCT99}.

\subsection{BCCB preconditioner} \label{sect:BCCBprec}
Let $C$ be the matrix obtained imposing periodic BCs. As described in Section \ref{ssect:BC}, 
the matrix $C$ is diagonalizable by DFT.
Hence, the matrix-vector product with the matrix
\[P=(CC^T+\alpha I)^{-1}\]
can be efficiently computed by FFT and its use  was previously proposed in \cite{COS}.

\begin{algorithm}
\begin{equation}\label{eq:Alg1}
    \left\{
    \begin{array}{l}
          \zb^{n+1}=\zb^{n}+WA^T(CC^T+\alpha I)^{-1}(\gb-AW^T\xb^{n}), \\
         \xb^{n+1}=\S_\mu(\zb^{n+1}).
    \end{array}
    \right.
\end{equation}
\end{algorithm}

We suggest to replace the DFT with the DCT, in the case of a quadrantally symmetric PSF like the Gaussian blur. 
The latter choice not only is motivated by computational considerations, since the complex DFT is replaced by a real transform (DCT), but also by the quality of the computed approximation. Indeed, using the DCT the matrix $C$ can be seen as an approximation of $A$ by imposing reflective BCs instead of periodic BCs, which results usually in a better approximation and so provides better restorations. The same expedient will be used for the following preconditioners as well.

Note that
$\|(CC^T + \alpha I)^{-1}\|<\|(AA^T + \alpha I)^{-1}\|$
is a sufficient condition to apply the Theorem \ref{lem:modLinBregConv} with $\lambda=1$.
This assumption could be hard to be satisfied in practice. Nevertheless, it is expected that 
\begin{equation}\label{eq:condconvalg1}
\|K^T(CC^T + \alpha I)^{-1}K\|<1
\end{equation}
if $\alpha$ is large enough.

In the literature regarding preconditioning of Toeplitz matrices by circulant matrices, several strategies have been proposed to compute the matrix $C$, cf. \cite{ChanNg}. In this paper we simply consider the matrix obtained imposing periodic BCs, which corresponds to the natural Strang preconditioner, since we have not observed numerical differences, when using other strategies like the optimal Frobenius norm approximation preconditioner \cite{OptC}. Roughly speaking, this follows from the fact that the
entries of $A$ depend on the value of the pixels of the PSF according to the shift invariance structure. In particular the central coefficient of the PSF belongs to the main diagonal of $A$ and the pixels near the center of the PSF belongs to the central diagonals of the central blocks. Finally, the PSF is almost centered in the middle of a $n \times n$ image and hence every pixel is distant at most $n/2$ pixels in every direction (usually much less due to the compact support).
Hence the block band and the band of each block of $A$ are at most $n/2$.

\subsection{Krylov subspace approximation} \label{sect:Krylovprec}

The preconditioner in the previous section is essentially defined as an approximation of the operator $A$ in the Fourier domain.
Another strategy, useful also for more general matrices, is to employ orthogonal or oblique projections into subspaces of small dimension.
A common choice is a proper Krylov subspace.

The matrix vector product $\tb^{n}=(AA^T + \alpha I)^{-1}\rb^{n}$
can be computed solving the linear system
\begin{equation}\label{eq:syscg}
(AA^T + \alpha I)\tb^{n}=\rb^{n},
\end{equation} 
whose solution can be approximated by few iterations of conjugate gradient (CG) since $AA^T + \alpha I$ is symmetric and positive definite. One or few steps of CG to approximate the vector $(AA^T + \alpha I)^{-1}\rb^{n}$ is a common strategy, see e.g. \cite{AF13}.
Here we explore the use of a good preconditioner associated with a proper choice of the initial guess and the stopping criteria. We solve \eqref{eq:syscg} by preconditioned CG (PCG)
with preconditioner the matrix $(CC^T + \alpha I)^{-1}$ introduced in Section \ref{sect:BCCBprec}. 
This is equivalent to solve the linear system
\begin{equation}\label{eq:psys}
(CC^T + \alpha I)^{-1/2}(AA^T + \alpha I)(CC^T + \alpha I)^{-1/2}\yb^{n}=(CC^T + \alpha I)^{-1/2}\rb^{n},
\end{equation}
with $\yb^{n} = (CC^T + \alpha I)^{1/2}\tb^{n}$.

The Krylov subspace of size $j$ generated by the matrix $B$ and the vector $\vb$ is defined by
\[
\mathcal{K}_j(B,\vb)={\rm span}\{\vb, B\vb, \dots, B^{j-1}\vb\}, \qquad j \in \ \N.
\]
We denote by $\yb^{n}_{\beta_n}$ the vector that minimizes the energy norm of the error of the linear system \eqref{eq:psys} into the Krylov subspace
\[ \mathcal{K}_{\beta_n} := \mathcal{K}_{\beta_n}\left((CC^T + \alpha I)^{-1/2}(AA^T + \alpha I)(CC^T + \alpha I)^{-1/2}, \; (CC^T + \alpha I)^{-1/2}\rb^{n}\right).\]
Therefore, defining
$$\tb^{n}_{\beta_n}=(CC^T + \alpha I)^{-1/2}\yb^{n}_{\beta_n}$$
the following algorithm can be sketched.

\begin{algorithm}
\begin{equation}\label{eq:Alg2}
    \left\{
    \begin{array}{l}
          \zb^{n+1}=\zb^{n}+WA^T\tb^{n}_{\beta_n}, \\
         \xb^{n+1}=\S_\mu(\zb^{n+1}).
    \end{array}
    \right.
\end{equation}
\end{algorithm}

Of course a large $\beta_n$ is not practical and also the convergence of the Algorithm 2 could fail 
if $\tb^{n}_{\beta_n}$ is not a good approximation of $\tb^n$ in \eqref{eq:syscg}.
However, in practice $\beta_n$ can be taken very small and the PCG converges very rapidly assuring also the convergence of the Algorithm~2 as numerically confirmed by the results in Section \ref{sec:numres}. This follows from 
discussion at the beginning of Section~\ref{ssect:BC} and the well-conditioning of the coefficient matrix of the linear system \eqref{eq:syscg}. Indeed, all the eigenvalues of $AA^T+\alpha I$ are in $[\alpha, \, c]$, with $c$ constant independent of $n$ and usually $c=1+\alpha$, because $A$ arises from the discretization of \eqref{eq:model2} and, thanks to the physical properties of the PSF (nonnegative entries and sum of all pixels equal to one), its largest singular value is bounded by one.
If follows that the BCCB preconditioner $CC^T + \alpha I$ is very effective since the spectrum of $(CC^T + \alpha I)^{-1/2}(AA^T + \alpha I)(CC^T + \alpha I)^{-1/2}$ is clustered at 1, with $O(n)$ outliers according to equations \eqref{eq:apprxAP} and \eqref{eq:ranknorm}, while $n^2$ is the total number of eigenvalues (see \cite{HN96,ChanBook}).

Moreover, we observe that if the Algorithm 2 is converging then, in the noise free case, the residual is going to zero (otherwise stagnates around $\delta$) and hence also the solution of the linear system \eqref{eq:syscg} approaches the zero vector. This has two interesting consequences. First, a good initial guess for the PCG is the zero vector since it is a good approximation of the solution of the linear system \eqref{eq:syscg}, at least for $n$ large enough. Second, the size of the Krylov subspace $\mathcal{K}_{\beta_n}$ should decreases when $n$ increases reaching the same fixed accuracy in the approximation of $\tb^n$ (see the following discussion on $\beta_n$).

Note that the computation of $\tb^{n}_{\beta_n}$ requires
$\beta_n$ matrix-vector products with $AA^T + \alpha I$ and with $CC^T + \alpha I$.
Nevertheless, according to the previous discussion, a small $\beta_n$, e.g., $\beta_n \leq 5$, is enough.
In the numerical results in Section~\ref{sec:numres} we fix
\begin{equation}\label{eq:bet}
\beta_n = \min\{5,\beta_{\rm tol}\},
\end{equation}
where $\beta_{\rm tol}$ is the number of PCG iterations required for reaching the tolerance $10^{-3}$ in
terms of the norm of the relative residual in the linear system \eqref{eq:psys}.
We observe that in our numerical results, $\beta_n$ decreases quickly obtaining $\beta_n=1$ for all $n>\bar{n}$, with
$\bar{n}$ small.

Finally, we note that the preconditioner $P$ obtained by the PCG approximation is not stationary and changes at each iteration of the MLBA. Therefore, Theorem~\ref{lem:modLinBregConv} cannot be applied. Nevertheless, Algorithm~2 with the condition \eqref{eq:bet} has been convergent in all our numerical experiments confirming that 
$\tb^{n}_{\beta_n}$ is a good approximation of $\tb^{n}$, at least for $n$ large enough.

\subsection{Preconditioning by symmetrization} \label{sec:simm}

The preconditioner in Section \ref{sect:BCCBprec} is related to a different boundary model, namely periodic BCs, but the deblurring problem and in particular the PSF are the same. Unfortunately, periodic BCs lead to poor restorations for generic images and for more accurate models, like reflective or antireflective BCs, the matrix $A$ cannot be diagonalized by fast trigonometric transforms when the PSF is not quadrantally symmetric.
Therefore, in this section we use a different strategy: the preconditioner is defined by a different PSF that leads to fast computations with an accurate deblurring model.

We consider a simple implementation of this strategy that can be useful when the PSF is experimentally measured. Indeed, in some applications, the PSF is nonsymmetric even if it is just a numerical perturbation of a Gaussian-like blur, cf. Example 1 and \cite{HN96}. Recalling that for the reflective and antireflective BCs fast transforms (cosine and antireflective, respectivelly) can be implemented only in the quadrantally symmetric case, a quadrantally symmetric PSF $\widetilde H$ can be obtained from the original PSF $H$ by defining
\[ \widetilde{H}(i,j)=\frac{H(i,j) + H(-i,j) + H(i,-j) + H(-i,-j)}{4}, \qquad i,j=1,\dots,n. \]
Note that $\widetilde{H}$ is the optimal Frobenius norm approximation of $H$ in the set of quadrantally symmetric PSFs, see\cite{NCT99}.
Therefore, we consider the matrix $Q$ obtained imposing reflective BC to $\widetilde H$ when $A$ is the BTTB or the rectangular matrix, while for $A$ antireflective, $Q$ is defined imposing antireflective BCs as well. In this way
\[ P = (QQ^T+\alpha I)^{-1} \]
can be diagonalized by DCT or by antireflective transform and the MLBA becomes

\begin{algorithm}
\begin{equation}\label{eq:Alg3}
    \left\{
    \begin{array}{l}
          \zb^{n+1}=\zb^{n}+WA^T(QQ^T+\alpha I)^{-1}(\gb-AW^T\xb^{n}), \\
         \xb^{n+1}=\S_\mu(\zb^{n+1}).
    \end{array}
    \right.
\end{equation}
\end{algorithm}

In analogy to Algorithm 1,
$\|(QQ^T + \alpha I)^{-1}\|<\|(AA^T + \alpha I)^{-1}\|$
is a sufficient condition to apply Theorem \ref{lem:modLinBregConv} with $\lambda=1$.
It is expected that 
\[
\|K^T(QQ^T + \alpha I)^{-1}K\|<1
\]
if $\alpha$ is large enough.

\section{Approximated Tikhonov regularization instead of preconditioning}\label{sect:alg4}

In this section we propose an approach to approximate $K^\dag$ different from the use of the matrix $P$ 
in \eqref{eq:P} 
as suggested in \cite{COS}.
Motivated by a very recent preconditioning proposal in 
\cite{Acqua,DH13},
we replace the whole matrix $K^\dag$ by a regularized approximation obtained by $C$.

In \cite{Acqua} the authors suggest to solve the preconditioned linear system
\[
ZA\fb=Z\gb,
\] 
where $Z$ is a regularized approximation of $K^\dag$, by a Van Cittert iteration \cite{VanC},
instead to solve a preconditioned Landweber iteration. Unfortunately, $ZA$ is not symmetric and the convergence analysis, based on the complex eigenvalues of $ZA$, is hard to generalize.
Differently, the nonstationary preconditioned iteration proposed in \cite{DH13} results in a similar iteration, but an elegant convergence analysis is provided under a minor approximation assumption. 

For $P=(AA^T + \alpha I)^{-1}$, the correction term $K^TP( \gb- K\xb^{n})$ in the MLBA \eqref{eq:MLBA}
can be seen as the Tikhonov solution, with parameter $\alpha$, of the error equation.
In detail, $\zb^{n+1}$ in the iteration \eqref{eq:MLBA} can be rewritten as
\begin{equation}\label{eq:itcorr}
          \zb^{n+1}=\zb^{n}+\pb^{n}, \\
\end{equation}
where
\begin{align*}
\pb^{n} &= K^TP( \gb- K\xb^{n}) \\
&= K^T(KK^T + \alpha I)^{-1}( \gb- K\xb^{n})\\
&=(K^TK + \alpha I)^{-1}K^T( \gb- K\xb^{n}),
\end{align*}
since
$(AA^T + \alpha I)^{-1}=(KK^T + \alpha I)^{-1}$ and
$K^T(KK^T + \alpha I)^{-1} = (K^TK + \alpha I)^{-1}K^T$.
Note that the correction $\pb^{n}$ is the solution of the Tikhonov problem
\[
\min_{\pb \in \R^s}\{\|K\pb - \rb^{n}\|^2 + \alpha \|\pb\|^2\}, 
\]
which is a regularized approximation of the error equation
\begin{equation}\label{eq:err}
K\eb^{n}=\rb^{n}
\end{equation}
in the noise free case, i.e., $\delta=0$,
where $\eb^{n} = \xb - \xb^{n}$ denotes the error at the current iteration.

In real applications $\delta \neq 0$ and so
 equation \eqref{eq:err} is (only) correct up to
the perturbation in the data.
Taking this into account, one may as well
consider instead of the error equation~\eqref{eq:err} the ``model equation''
\begin{equation}\label{eq:erreq}
   L\eb^{n}=\rb^{n}\,,
\end{equation}
where $L$ is an approximation of $K$,
possibly tolerating a slightly larger misfit. Solving \eqref{eq:erreq}
by means of Tikhonov regularization, we find
\begin{align*}
\widetilde \pb^{n} &= (L^TL + \alpha I)^{-1}L^T\rb^{n}\\
&= WC^T(CC^T + \alpha I)^{-1}\rb^{n},
\end{align*}
where we have choosen
\begin{equation}\label{eq:L}
L=CW^T.
\end{equation}
Using $\widetilde \pb^{n}$ in \eqref{eq:itcorr} to replace $\pb^{n}$
we obtain a new algorithm.

\begin{algorithm}
\begin{equation}\label{eq:Alg4}
    \left\{
    \begin{array}{l}
          \zb^{n+1}=\zb^{n}+WC^T(CC^T+\alpha I)^{-1}(\gb-AW^T\xb^{n}), \\
         \xb^{n+1}=\S_\mu(\zb^{n+1}).
    \end{array}
    \right.
\end{equation}
\end{algorithm}

As before, the matrix $C$ is chosen as a BCCB in general or diagonalizable by DCT (i.e., the reflective BC matrix), when the PSF is quadrantally symmetric.
Unfortunately, this preconditioning strategy cannot be applied to the rectangular matrix approach
because, in such case, $C$ should have the same size of $A$, but this condition prevents the possibility
of computing $\widetilde \pb^{n}$ by fast trigonometric transforms.

\begin{remark}\label{rem:prec2}
The iteration \eqref{eq:Alg4} uses the preconditioned linear system
\begin{equation}\label{eq:precsys}
WC^T(CC^T+\alpha I)^{-1}K\xb=WC^T(CC^T+\alpha I)^{-1}\gb,
\end{equation}
to update an aproximation inspired by the linearized Bregman iteration \eqref{eq:iterLinBreg}, but without resorting to the normal equations.
\end{remark}

Clearly, Algorithm 4 is no longer a MLBA and so a different convergence analysis is required. 
Unfortunately, classical results for convex optimization cannot be applied since
the coefficient matrix $WC^T(CC^T+\alpha I)^{-1}K$ in \eqref{eq:precsys} is not symmetric positive definite.
An alternative convergence proof could be very hard because also the complex analysis convergence in 
 \cite{Acqua} cannot be easily combined with the Bregman splitting and soft-thresholding.

Therefore, accordingly to \cite{DH13}, we consider a nonstationary choice of $\alpha$ that allows to provide a convergence analysis of the resulting algorithm and avoid the a-priori choice of $\alpha$. On the other hand, if a good estimation of $\alpha$ is available, then Algorithm 4 can provides better restorations and hence it is also considered in the numerical results in Section \ref{sec:numres}.

\begin{assumption}\label{hp:1}
Let $A,C\in\R^{n^2 \times n^2}$ and $W\in\R^{n^2 \times s}$, $n^2\leq s$, such that
\begin{subequations}\label{eq:Alg4_rev_00}
\begin{equation}\label{eq:Alg4_rev_0}
\| \left( C - A \right) \vb \| \leq \rho \| A \vb \|, \qquad \forall \, \vb \in \R^{n^2},
\end{equation}
and
\begin{equation}\label{eq:Alg4_rev_01}
\| CW^T(\ub- S_\mu(\ub))\| \leq \rho \delta, \qquad \forall \, \ub \in \R^s,
\end{equation}
\end{subequations}
with a fixed $0 < \rho < 1/2$,
where $\delta=\|\etab\|$ is the noise level.
\end{assumption}

The assumption \eqref{eq:Alg4_rev_0} is the same spectral equivalence required in \cite{DH13}.
Let $L$ be defined in \eqref{eq:L}, then
equation \eqref{eq:Alg4_rev_0} translates into
\begin{equation}\label{eq:Alg4_rev_1}
\| ( L - K ) \ub \| \leq \rho \| K \ub \|, \qquad \forall \, \ub \in \R^s\!.
\end{equation}
Instead, the assumption \eqref{eq:Alg4_rev_01} was not present in \cite{DH13}
and it is equivalent to consider the
soft-threshold parameter $\mu$ as a continuous function with respect to the noise level $\delta$, i.e., $\mu = \mu (\delta)$, and such that $\mu (\delta) \to 0$ as $\delta \to 0 $. This is a common request in many soft-thresholding based methods, see for instance Theorem 4.1 in \cite{D3}.
Nevertheless, in this paper we will not concentrate on $\mu$ and will not give any specific $\delta$-dependent rule to compute it. 

\begin{algorithm*}  \textbf{\emph{4--NS.}}
Let $\zb^0$ be given and set $\rb^0 = \gb - K S_\mu(\zb^0)$. 
 Choose $\tau = \frac{1+ 2\rho}{1- 2\rho}$ 
 with $\rho$ from \eqref{eq:Alg4_rev_0}, and fix $q \in (2\rho, 1)$.

While $\| \rb^n \| > \tau \delta$, let $\tau_n = \|\rb^n \|/ \delta$ and let $\alpha_n$ be such that
\begin{subequations}
\begin{eqnarray}\label{eq:Alg4_rev_7}
\alpha_n \| (CC^T + \alpha_n I)^{-1} \rb^n\| = q_n \|\rb^n\|, \qquad q_n = \max \{ q, 2\rho + (1+\rho)/\tau_n \},
\end{eqnarray}
compute 
\begin{equation}\label{eq:Alg4_revised}
    \left\{
    \begin{array}{l}
          \zb^{n+1}=\zb^{n}+ WC^T(CC^T+\alpha_n I)^{-1}(\gb - AW^T\xb^{n}), \\
         \xb^{n+1}=\S_\mu(\zb^{n+1}).
    \end{array}
    \right.
\end{equation}
\end{subequations}
\end{algorithm*}

Note that the iteration \eqref{eq:Alg4_revised} is the same of Algorithm 4 where a nonstationary $\alpha$ is chosen at every 
iteration according to \eqref{eq:Alg4_rev_7}.
In Corollary \ref{cor:Alg4_rev_1} we will prove that, if $\delta >0$, then Algorithm 4--NS will terminate after $n = n_\delta \geq 0$ iterations with
\begin{equation}\label{eq:Alg4_rev_12}
\| \rb^{n_\delta} \| \leq \tau \delta < \| \rb^{n} \|, \qquad n= 0,1, \cdots, n_\delta -1,
\end{equation}
which is the \textnormal{discrepancy principle} \eqref{eq:discrpr} with $\tilde{n}=n_\delta$ and  $\gamma=\tau = (1+2\rho)/(1-2\rho)$.

The parameter $q$ in Algorithm 4, like in \cite{DH13}, is meant as a safeguard to prevent that the residual decreases too rapidly. Our theoretical results do not utilize this parameter. 

\begin{remark}\label{remark:Alg4_rev_1}
It is not difficult to see that there is a unique positive parameter $\alpha_n$ that satisfies \eqref{eq:Alg4_rev_7}. This parameter can be computed with a few step of an appropriate Newton scheme~\cite{EHN96}. Accordingly, parameter $\alpha_n$, and therefore Algorithm 4--NS, are well defined.
\end{remark}

We define 
\begin{equation}\label{eq:hath}
\hb^n = L^T (LL^T + \alpha_n I)^{-1} (\gb - K\S_\mu (\zb^{n})),
\end{equation}
such that \eqref{eq:Alg4_revised} can be compactly rewritten as
\[
    \left\{
    \begin{array}{l}
          \zb^{n+1}=\zb^{n}+ \hb^n, \\
         \xb^{n+1}=\S_\mu(\zb^{n+1}).
    \end{array}
    \right.
\]

For the purpose of the subsequent convergence and regularization results, when $\delta >0$, even if it will be always the case, we will highlight by the subscript $\delta$ (for instance $\{\xb_\delta^n\}$) the sequences generated by Algorithm 4--NS, starting from initial data $\gb ^\delta = A\fb + \etab$ affected by noise, whereas we avoid the subscript (for instance $\{\xb ^n\}$) for the sequences generated starting from exact initial data $\gb = A\fb$, i.e., $\delta=0$.

For the following analysis, instead of working with the error 
$\eb^n _\delta = \xb - \xb^n _\delta$,
it is useful to consider the partial error with respect to $\zb^n _\delta$, namely
\begin{equation}\label{eq:tildee}
\tilde{\eb}^n _\delta = \xb - \zb^n _\delta.
\end{equation}

\begin{proposition}\label{prop:Alg4_rev_1}
Assume that the assumptions \eqref{eq:Alg4_rev_00}
are satisfied for some $0 < \rho < 1/2$. 
If $\| \rb^n _\delta\| > \tau \delta$ and we define $\tau_n = \| \rb^n_\delta \|/ \delta$, then it follows that
\begin{equation}\label{eq:Alg4_rev_10}
\| \rb^n _\delta - L\tilde{\eb}^n _\delta \| \leq \left(  \rho + \frac{1+ 2\rho}{\tau_n}  \right) \| \rb^n _\delta \| 
< (1-\rho) \| \rb^n _\delta \|,
\end{equation}
where $\tilde{\eb}^n$ is defined in \eqref{eq:tildee}.
\end{proposition}
\begin{proof}
In the free noise case we have $\gb = K\xb$. As a consequence
\begin{align*}
\rb^n _\delta - L\eb^n _\delta & = \gb^\delta - K \xb^n _\delta -  L(\xb - \zb^n _\delta) + L \xb^n _\delta - L S_\mu(\zb^n _\delta) \\
	& = \gb^\delta - \gb + (K - L)\eb^n _\delta + L(\zb^n _\delta-S_\mu(\zb^n _\delta)).
\end{align*}
Using now assumptions \eqref{eq:Alg4_rev_00}, in particular 
\eqref{eq:Alg4_rev_1}, and $\| \gb^\delta - \gb \| \leq \delta$, we derive the following estimate
\begin{align*}
\| \rb^n _\delta - L\eb^n _\delta \| &\leq  \| \gb^\delta - \gb \| + \| (K - L)\eb^n _\delta  \| + \|L(\zb^n _\delta-S_\mu(\zb^n _\delta))\|\\
&\leq  \| \gb^\delta - \gb \| + \rho \| K\eb^n _\delta \| +\rho\delta\\
												  & \leq  \| \gb^\delta - \gb \| + \rho (\|\rb^n _\delta \| + \| \gb^\delta - \gb \|  + \delta)\\
												  &\leq (1 + 2\rho) \delta + \rho \|\rb^n _\delta \|.
\end{align*}
The first inequality in \eqref{eq:Alg4_rev_10} now follows from the hypothesis $\delta = \|\rb^n _\delta \|/ \tau_n$.
The second inequality  follows from $ \rho + \frac{1+ 2\rho}{\tau_n} < \rho + \frac{1+ 2\rho}{\tau}$.
\end{proof}

We are going to show that the sequence $\{ \xb^n _\delta\}$ approaches $\xb$ as $\delta \to 0$. The proof combines Proposition~\ref{prop:Alg4_rev_1} with suitable modifications of the results in \cite{DH13}.

\begin{proposition}\label{prop:Alg4_rev_2}
Let $\tilde{\eb}^n _\delta$ be defined in \eqref{eq:tildee}. 
If the assumptions \eqref{eq:Alg4_rev_00} are satisfied, then $\|\tilde{\eb}^n _\delta\|$ of Algorithm 4--NS decreases monotonically for $n=0,1,\dots,n_\delta-1$. In particular, we deduce
\begin{equation}\label{eq:Alg4_rev_6}
\| \tilde{\eb}^n _\delta \|^2 - \| \tilde{\eb}^{n+1} _\delta \|^2 \geq \frac{8\rho^2}{1 + 2\rho} \| (CC^T + \alpha_n )^{-1} \rb^n _\delta \| \|\rb^n  _\delta\|.
\end{equation}
\end{proposition}
\begin{proof}
We have
\begin{align*}
\| \tilde{\eb}^n _\delta \|^2 - \| \tilde{\eb}^{n+1} _\delta \|^2 &= 2\langle\tilde{\eb}^n _\delta, \hb^n \rangle - \| \hb^n \|^2\\
&= 2\langle L\tilde{\eb}^n _\delta, (CC^T + \alpha_n I)^{-1} \rb^n _\delta \rangle - \langle  \rb^n _\delta, CC^T (CC^T + \alpha_n I)^{-2}  \rb^n _\delta \rangle \\
&= 2 \langle  \rb^n _\delta, (CC^T + \alpha_n I)^{-1}  \rb^n _\delta \rangle -   \langle  \rb^n _\delta, CC^T (CC^T + \alpha_n I)^{-2}  \rb^n _\delta \rangle \\
&- 2 \langle  \rb^n _\delta - L\tilde{\eb}^n _\delta, (CC^T + \alpha_n I)^{-1}  \rb^n _\delta \rangle\\
&\geq 2 \langle  \rb^n _\delta, (CC^T + \alpha_n I)^{-1}  \rb^n _\delta \rangle - 2  \langle  \rb^n _\delta, CC^T (CC^T + \alpha_n I)^{-2}  \rb^n _\delta \rangle \\
&- 2 \langle  \rb^n _\delta - L\tilde{\eb}^n _\delta, (CC^T + \alpha_n I)^{-1}  \rb^n _\delta \rangle\\
&= 2 \alpha_n \langle  \rb^n _\delta, (CC^T + \alpha_n I)^{-2}  \rb^n _\delta \rangle - 2 \langle  \rb^n _\delta - L\tilde{\eb}^n _\delta, (CC^T + \alpha_n I)^{-1}  \rb^n _\delta \rangle\\
&\geq  2 \alpha_n \langle  \rb^n _\delta, (CC^T + \alpha_n I)^{-2}  \rb^n _\delta \rangle - 2 \|  \rb^n _\delta - L\tilde{\eb}^n _\delta \| \| (CC^T + \alpha_n I)^{-1}  \rb^n _\delta \|\\
&= 2 \| (CC^T + \alpha_n I)^{-1}  \rb^n _\delta \| \left( \| \alpha_n  (CC^T + \alpha_n I)^{-1}  \rb^n _\delta\| - \| \rb^n _\delta - L\tilde{\eb}^n _\delta \|  \right)\\
&\geq  2 \| (CC^T + \alpha_n I)^{-1}  \rb^n _\delta \| \left( q_n \| \rb^n _\delta\|  - \left( \rho + \frac{1+ 2\rho}{\tau_n}  \right) \| \rb^n _\delta\|  \right)\\
&\geq \frac{8\rho^2}{1 + 2\rho}  \| (CC^T + \alpha_n I)^{-1}  \rb^n _\delta \|\| \rb^n _\delta\|,
\end{align*}
where the relevant inequalities are a consequence of equation \eqref{eq:Alg4_rev_7} and Proposition~\ref{prop:Alg4_rev_1}. 
The last inequality follows from \eqref{eq:Alg4_rev_7} and $\tau_n > \tau=(1+2\rho)/(1-2\rho)$ for $\| \rb^n _\delta \| > \tau \delta$. 
\end{proof}

\begin{corollary}\label{cor:Alg4_rev_1}
Under the assumptions \eqref{eq:Alg4_rev_00}, there holds
\begin{equation}\label{eq:Alg4_rev_11}
\| \tilde{\eb}^0 _\delta \| \geq  \frac{8\rho^2}{1 + 2\rho} \sum_{n=0}^{n_\delta -1}  \| (CC^T + \alpha_n I)^{-1}  \rb^n _\delta \|\| \rb^n _\delta\| \geq c  \sum_{n=0}^{n_\delta -1}\| \rb^n _\delta\|^2
\end{equation}
for some constant $c >0$, depending only on $\rho$ and $q$ in \eqref{eq:Alg4_rev_7}.
\end{corollary}
\begin{proof}
The following proof is almost the same as Corollary 3 in \cite{DH13}, but we include it to make the paper self contained.

The first inequality follows by taking the sum of the quantities in \eqref{eq:Alg4_rev_6} from $n=0$ up to $n= n_\delta -1$. 

For the second inequality, note that for every $\alpha > \frac{q_n \| C \|^2 }{1 - q_n}$ and every $\sigma \in \sigma(C) \subset [0,\| C \|^2]$, with $\sigma(C)$ being the spectrum of $C$, we have
$$
\frac{\alpha}{\sigma^2 + \alpha} \geq \frac{\alpha}{\| C \|^2 + \alpha} = (1 + \|C\|^2/\alpha)^{-1} > q_n,
$$
and hence,
$$
\alpha \| (CC^T + \alpha I)^{-1} \rb^n _\delta \| > q_n \|  \rb^n _\delta \|,
$$
as $\|  \rb^n _\delta \|>0$ for $n < n_\delta$. This implies that  $\alpha_n$ in \eqref{eq:Alg4_rev_7} satisfies $0~< ~\alpha_n ~\leq ~\frac{q_n \|C \|^2}{1 - q_n}$, thus
$$
\| (CC^T + \alpha_n I)^{-1}  \rb^n _\delta \| = \frac{q_n}{\alpha_n} \|  \rb^n _\delta \| \geq \frac{(1 - q_n)}{\| C\|^2} \| \rb^n _\delta \|.
$$
Now, according to the choice of parameters in Algorithm 4-NS, we deduce \\$1 - q_n = \min \{1 - q, 1 - 2\rho - (1 + \rho)/\tau_n   \}$, and 
$$
1 - 2\rho - (1 + \rho)/\tau_n = \frac{1 + 2\rho}{\tau} - \frac{1 + \rho}{\tau_n} > \frac{1 + 2\rho}{\tau} - \frac{1 + \rho}{\tau} = \frac{\rho}{\tau}.
$$
Therefore, there exists $c >0$, depending only on $\rho$ and $q$ such that $1 - q_n \geq c \|C \|^2  \left(\frac{8\rho^2}{1 + 2\rho}\right)^{-1}$, and
$$
\| (CC^T + \alpha_n I)^{-1}  \rb^n _\delta \| \geq c  \left(\frac{8\rho^2}{1 + 2\rho}\right)^{-1} \|  \rb^n _\delta \| \qquad \mbox{for } n= 0, 1, \cdots, n_\delta -1.
$$
Now the second inequality follows immediately.
\end{proof}

From the outer inequality of \eqref{eq:Alg4_rev_11} it can be seen that the sum of the squares of the residual norms is bounded, and hence, if $\delta >0$, there must be a first integer $n_\delta < \infty$ such that \eqref{eq:Alg4_rev_12} is fulfilled, i.e., Algorithm 4-NS terminates after finitely many iterations. 

\begin{remark}\label{remark:Alg4_rev_2}
Recalling that 
the soft-threshold parameter $\mu$ is taken as a continuous function with respect to the noise level $\delta$ such that $\mu (\delta) \to 0$ as $\delta \to 0 $, then the operator $\gb \mapsto \zb^n$ is continuous for every fixed $n$.
\end{remark}

In the next theorem we are going to give a convergence and regularity result.
\begin{theorem}
Assume that $\zb^0$ is not a solution of the linear system
\begin{equation}\label{eq:O}
\gb = A W^T \xb,
\end{equation}
 and that $\delta_m$ is a sequence of positive real numbers such that $\delta_m \to 0$ as $m \to \infty$. Then, the sequence $\{ \xb^{n(\delta_m)} _{\delta_m}  \}_{m \in \mathbb{N}}$, generated by the discrepancy principle rule \eqref{eq:Alg4_rev_12}, converges as $m \to \infty$ to the solution of \eqref{eq:O} which is closest to $\zb ^0$ in Euclidean norm.
\end{theorem}
\begin{proof}
We are going to show convergence for the sequence $\{ \zb^{n(\delta_m)} _{\delta_m}  \}_{m \in \mathbb{N}}$ and then the thesis will follow easily from the continuity of $S_{\mu(\delta)}$ and Remark \ref{remark:Alg4_rev_2}, i.e.,
$$
\lim_{m\to \infty} \xb^{n(\delta_m)} _{\delta_m} = \lim_{m\to \infty}S_{\mu(\delta_m)} (\zb^{n(\delta_m)} _{\delta_m}) =  S_{\lim_{m\to \infty}\mu(\delta_m)} (\lim_{m\to \infty}\zb^{n(\delta_m)} _{\delta_m}) = \lim_{m\to \infty}\zb^{n(\delta_m)} _{\delta_m}.
$$
The proof of the convergence for the sequence $\{ \zb^{n(\delta_m)} _{\delta_m}  \}$ can be divided into two steps: at step one, we show the convergence in the free noise case $\delta =0$. In particular, the sequence $\{ \zb^n \}$ converges to a solution of \eqref{eq:O} that is the closest to $\zb ^0$. At the second step, we show that given a sequence of positive real numbers $\delta_m \to 0$ as $m \to \infty$, then we get a corresponding sequence $\{\zb ^{n(\delta_m)} _{\delta_m} \}$ converging as $m \to \infty$.

Concerning the first step of the proof, we will not give details since it can be just copied from \cite{DH13}[Theorem 4]. Indeed, if $\delta =0$, from Remark \ref{remark:Alg4_rev_2} it follows that $\rb^n _\delta = \rb^n$, and the sequence $\{ \zb ^n \}$ coincides with the one generated by algorithm 1 in \cite{DH13}. We just say that the main ingredients are the convergence of the sequence $\|\eb ^n \|$ granted by Proposition~\ref{prop:Alg4_rev_2} and the convergence to $0$ of the sequence $\| \rb ^n \| \| (CC^T + \alpha_n I)^{-1} \rb^n \|$, since general term of a converging series from Corollary \ref{cor:Alg4_rev_1}. Moreover, in the free noise case the sequence $\{\zb ^n  \}$ will not stop, i.e., $n \to \infty$, since the discrepancy principle will not be satisfied by any $n$, in particular $n_\delta \to \infty$ for $\delta \to 0$. 

Hence, let $\xb$ be the converging point of the sequence $\{ \zb ^n \}$ and let $\delta_m >0$ be a sequence of positive real numbers converging to $0$. For every $\delta_m$, let $n = n(\delta_m)$ be the first positive integer such that \eqref{eq:Alg4_rev_12} is satisfied, whose existence is granted by Corollary \ref{cor:Alg4_rev_1}, and let $\{ \zb^{n(\delta_m)} _{\delta_m} \}$ be the corresponding sequence. For every fixed $\epsilon >0$, there exists $\overline{n} = \overline{n}(\epsilon)$ such that 
\begin{equation}\label{eq:Alg4_rev_13}
\| \xb - \zb ^n \| \leq \epsilon /2 \qquad \mbox{for every } n > \overline{n}(\epsilon),
\end{equation} 
and there exists $\overline{\delta}= \overline{\delta}(\epsilon)$ for which
\begin{equation}\label{eq:Alg4_rev_14}
\| \zb ^{\overline{n}} - \zb^{\overline{n}} _\delta \| \leq \epsilon/2 \qquad \mbox{for every } 0 < \delta < \overline{\delta},
\end{equation}
due to the continuity of the operator $\gb \mapsto \zb^n$ for every fixed $n$, see Remark \ref{remark:Alg4_rev_2}. Therefore, let us choose $\overline{m}= \overline{m}(\epsilon)$ large enough such that $\delta_m < \overline{\delta}$ and such that $n(\delta_m) > \overline{n}$ for every $m > \overline{m}$. Such $\overline{m}$ does exists since $\delta_m \to 0$ and $n_\delta \to \infty$ for $\delta \to 0$. Hence, for every $m > \overline{m}$, we have
\begin{align*}
\| \xb - \zb^{n(\delta_m)} _{\delta_m} \| &= \| \tilde{\eb} ^{n(\delta_m)} _{\delta_m} \| \\
	&\leq  \| \tilde{\eb} ^{\overline{n}} _{\delta_m} \| \\
	&= \| \xb - \zb^{\overline{n}} _{\delta_m} \| \\
	&\leq \| \xb - \zb^{\overline{n}} \|  +  \| \zb^{\overline{n}} - \zb^{\overline{n}} _{\delta_m} \| \leq \epsilon,
\end{align*}
where the first inequality comes from Proposition \ref{prop:Alg4_rev_2} and the last one from \eqref{eq:Alg4_rev_13} and \eqref{eq:Alg4_rev_14}.
\end{proof}

\section{Numerical results}\label{sec:numres}
In this section, we will show the numerical results for image deblurring using our proposed Algorithms 1--4.
We compare them with some available deblurring algorithms, which implement a proper treatment of the boundary artifacts.
In particular we consider two of the algorithms proposed in \cite{AF13}, namely FA-MD for the Frame-based analysis model and TV-MD for the Total Variation model,
and the Algorithm~\cite{BCDS14} called here FTVd since, in the case of nonsymmetric PSF, it reduces to an
implementation of the algorithm in \cite{WYYZ08} with the trick described in \cite{R05}.
The codes of the previous algorithms are available at the web-page of the authors and we use the default parameters and stop conditions. The regularization parameter is chosen by hand in order to provide the best restoration (see the following discussion).

Our tests were done by using MATLAB 7.11.0 (R2010b) with  floating-point precision about $2.22 \cdot 10^{-16}$ on a Lenovo laptop with Intel(R) Core(TM) i2 CPU 2.20 GHz and 2 GB memory.
Assuming that the noise level is available or easily estimated, we stop all Algorithms~1--4 using the discrepancy principle
\eqref{eq:discrpr} with $\gamma=10^{-15}$.
Algorithm 4--NS is stopped according to the modified discrepancy principle \eqref{eq:Alg4_rev_12}.
Moreover, for the Algorithm~4--NS we set 
\[q=0.5 \qquad  \text{ and } \qquad \rho=10^{-4}.\] 
Therefore, $q$ and $\rho$ do not need to be estimated.

The accuracy of the solution is measured by the PSNR value, which is defined as
\[  {\rm PSNR} = 20\log_{10}\frac{255 \cdot n}{\|\fb - \tilde \fb\|}, \]
with $\fb$ and $\tilde \fb$ being the original and the restored images in the FOV, respectively.
The initial guess of each algorithm is set to be the zero vector.

To estimate $\mu$ and $\alpha$, since they are mutually dependent and they are related to the preconditioner,
we fix a possible $\mu$ (usually the results are not very sensible varying $\mu$ if $\alpha$ is properly chosen)
and then the optimum $\alpha$, which gives the largest PSNR, is chosen by trial and error.
Possible strategies to estimate $\alpha$ will be investigated in future works.
Only for Algorithm4--NS we pay a slightly more attention in the choice of $\mu$ since this is the only parameter of the method.
Similarly, for all the other methods considered for comparison, the regularization parameter is chosen by trial and error, as the one leading to
the largest PSNR.

We take only the more appropriate BCs for each example. In particular, if the image has a black background, like in astronomical imaging, we consider zero BCs, while when the image is a generic
picture we use antireflective BCs. In the following the ``Algorithm $x$'' is denoted by ``Alg-BC$x$'' and ``Alg-Rect$x$'', when $A$ is obtained imposing BCs or is the rectangular matrix, respectively. We recall
that Algorithm 4 is available only for the BC approach.

\subsection{Linear B-spline framelets}
The tight-frame used in our tests is the piecewise linear B-spline framelets given in \cite{COS}.
Namely, given the masks
\[
b_0 = \frac{1}{4} \,[1, \;2, \;1], \quad
b_1 = \frac{\sqrt{2}}{4} \,[1, \;0, \;-1], \qquad
b_2 = \frac{1}{4} \,[-1, \;2, \;-1],
\]
we define the 1D filters of size $n\times n$ by imposing reflective BCs
\[
    B_0 = \frac{1}{4}\left[
      \begin{array}{ccccc}
        3 & 1 & 0 & \dots  &0 \\
        1 & 2 & 1 & & \\
        & \ddots & \ddots & \ddots &\\
        &  & 1 & 2 & 1\\
        0 & \dots & 0 & 1&3 \\
      \end{array}
    \right],
    \quad
    B_1 = \frac{1}{4}\left[
      \begin{array}{ccccc}
        1 & -1 & 0 & \dots  &0 \\
        -1 & 2 & -1 & & \\
        & \ddots & \ddots & \ddots &\\
        &  & -1 & 2 & -1\\
        0 & \dots & 0 & -1&1 \\
      \end{array}
    \right],
\]
and
\[
    B_2 = \frac{1}{4}\left[
      \begin{array}{ccccc}
        -1 & 1 & 0 & \dots  &0 \\
        -1 & 0 & 1 & & \\
        & \ddots & \ddots & \ddots &\\
        &  & -1 & 0 & 1\\
        0 & \dots & 0 & -1&1 \\
      \end{array}
    \right].
\]
The nine 2D filters are obtained by
\[ B_{i,j} = B_i \otimes B_j, \qquad i,j=0,1,2,\]
where $\otimes$ denotes the tensor product operator.
Finally, the corresponding tight-frame analysis operator is
\[
W = \left[
\begin{array}{c}
  B_{0,0} \\
  B_{0,1} \\
  \vdots \\
  B_{2,2}\\
\end{array}
\right].
\]

Throughout the experiments, the level of the framelet decomposition is 4 like in \cite{COS}
and the level of wavelet decomposition is the one used in FA--MD.

\begin{figure}
\centering
 \subfigure[true image]{
    \includegraphics[width=0.18\textwidth]{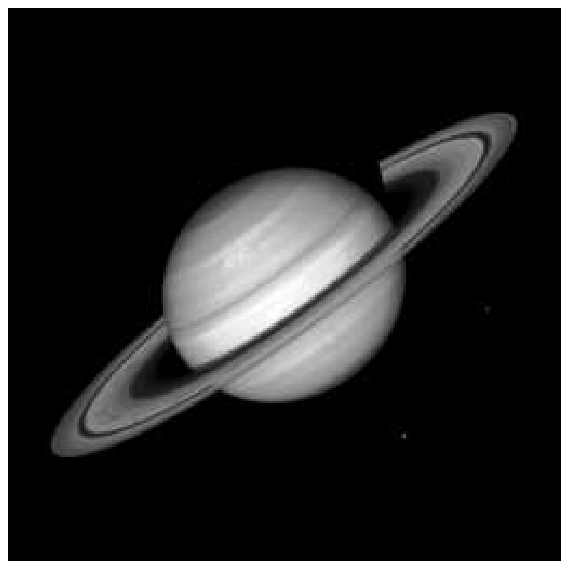}}
 \subfigure[PSF image]{
    \includegraphics[width=0.18\textwidth]{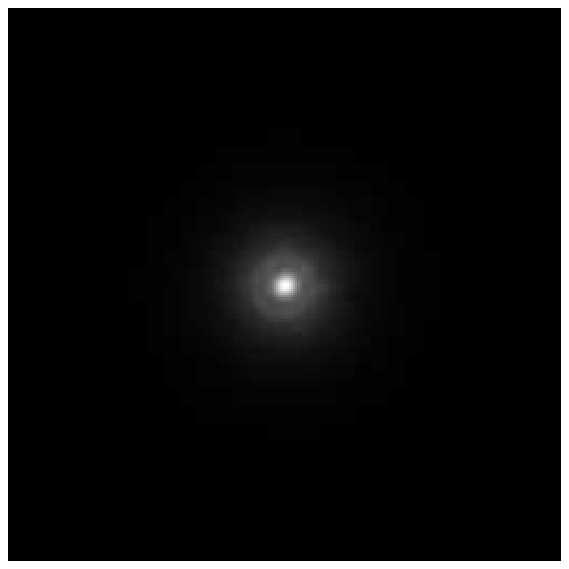}}
 \subfigure[observed image]{
    \includegraphics[width=0.18\textwidth]{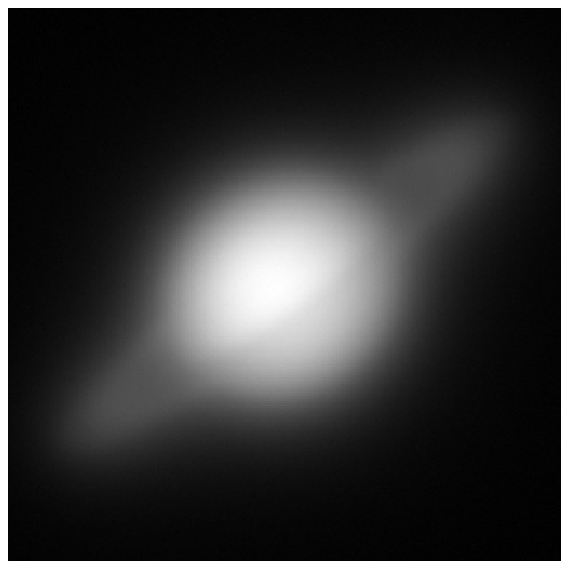}}
 \subfigure[Alg-BC1]{
    \includegraphics[width=0.18\textwidth]{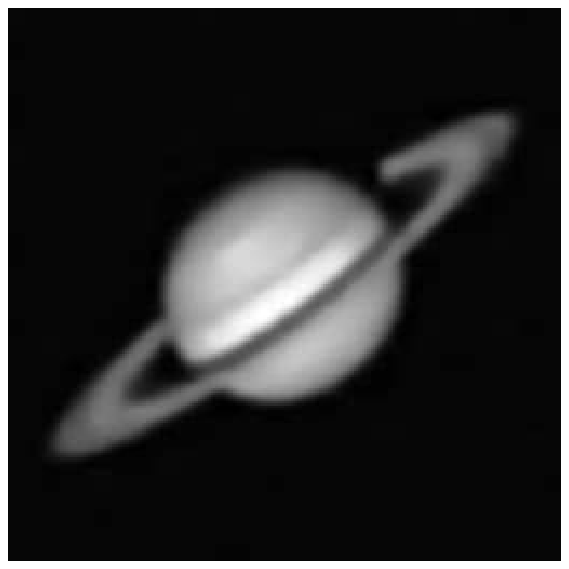}}
 \subfigure[Alg-BC3]{
    \includegraphics[width=0.18\textwidth]{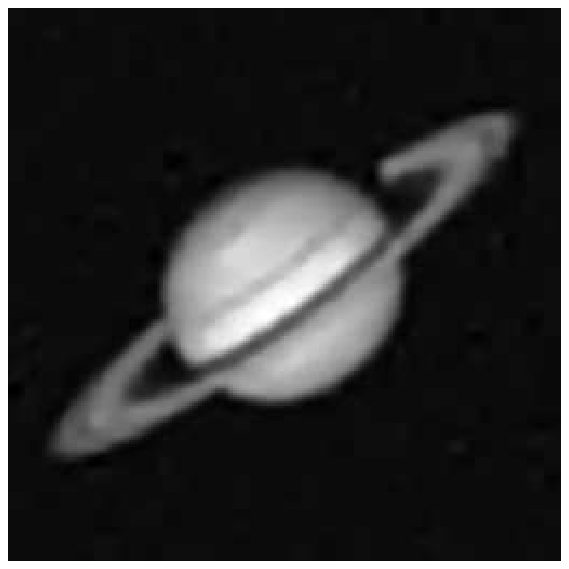}}
 \subfigure[Alg-Rect2]{
    \includegraphics[width=0.18\textwidth]{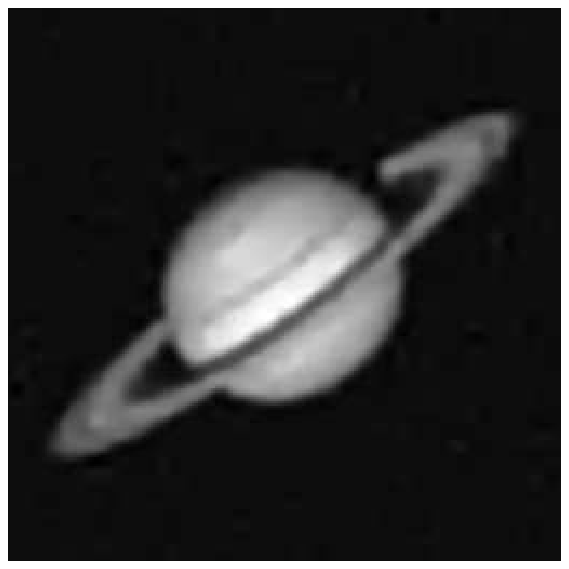}}
 \subfigure[Alg-Rect3]{
    \includegraphics[width=0.18\textwidth]{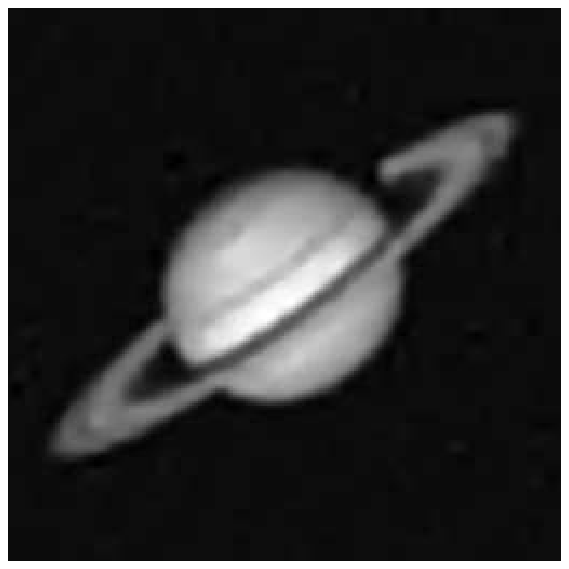}}
 \subfigure[Alg-BC4--NS]{
    \includegraphics[width=0.18\textwidth]{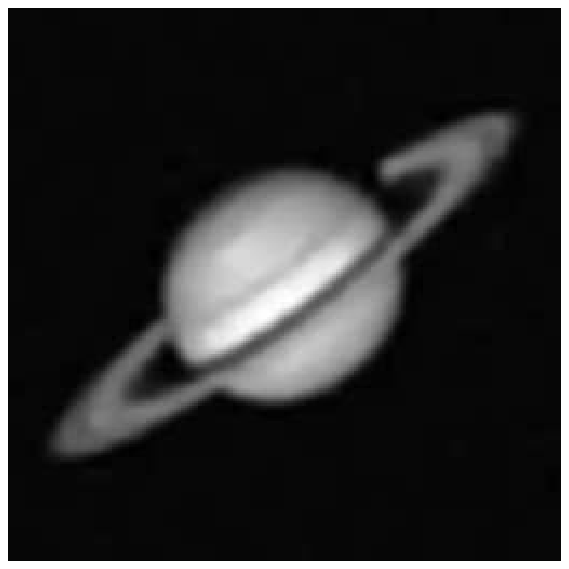}
}
 \subfigure[TV-MD]{
    \includegraphics[width=0.18\textwidth]{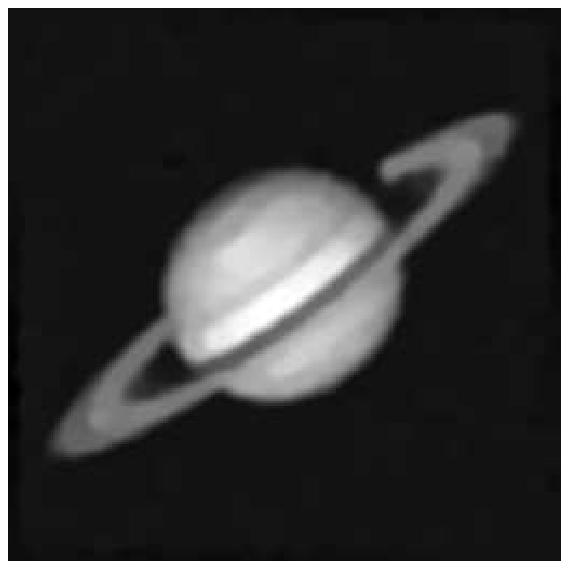}}
 \subfigure[FTVd]{
    \includegraphics[width=0.18\textwidth]{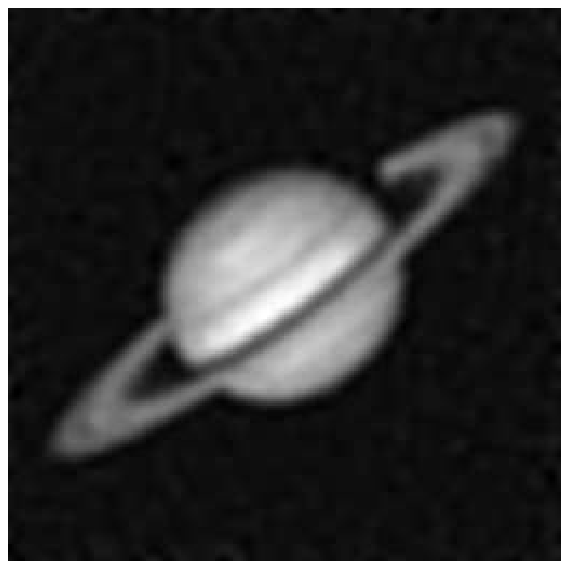}}
\caption{Example 1: true image, PSF, observed image, and restored images.}
\label{fig:Test1}
\end{figure}

\subsection{Example 1: Saturn image}
The first example is $256 \times 256$ Saturn image in Figure~\ref{fig:Test1}~(a) while
the astronomical PSF is taken from the ``satellite'' test problem in \cite{RestoreTools}  Figure~\ref{fig:Test1}~(b). We add a $1\%$ of Gaussian white noise to obtain the observed image in Figure~\ref{fig:Test1}~(c).
We assume zero BCs.

Note that Alg-BC3 and Alg-Rect3 use the DCT for the preconditioner, while Alg-BC4 like Alg-BC1 and Alg-Rect1 use FFT, since the PSF is not quadrantally symmetric.

Table \ref{tab:Test1} reports the PSNR and the CPU time for the different algorithms. Note that Algorithm~1 provides a poor and time consuming restoration. Moreover, it requires a larger value of the parameter $\alpha$ with respect to the algorithms~2--4, which is necessary to satisfy condition \eqref{eq:condconvalg1} and assure the convergence. 
The algortihms 2 and 3 have the largest PSNR with reasonable CPU time, in particular Alg-Rect3 seems to be a good choice and Alg-BC3 gives a comparable restoration in about half time.
Algorithm 4 gives a slightly lower PSNR even if the computed restorations are better than Algorithm 1 and the other algorithms from the literature, keeping also a low CPU time. 


The algorithms in \cite{AF13} (FA-MD and TV-MD) in this example lead to a larger CPU time, while the FTVd is very fast but the computed restoration is the worst.
Figure \ref{fig:Test1} shows the corresponding restored images.
To test the quality of the restorations, Figure~\ref{fig:Res1} shows the residual images defined as $\gb-A \tilde \fb$, where $ \tilde \fb$ is the restored image.

\begin{table}
\begin{center}\begin{small}
\begin{tabular}
{p{60pt}|crr|p{90pt}} \hline
Algorithm  & PSNR & Iter. & CPU time(s) & Regular. parameter  \\
\hline
\cline{1-5}
{Alg-BC1}  & 30.97  & 322    & 200.99  & $\alpha = 0.045$      \\
{Alg-BC2}   & 31.60  & 9    & 18.18 & $\alpha = 0.0004$ \\
{Alg-BC3}  & 31.56  & 10    & 7.07 & $\alpha = 0.0005$  \\
\cline{1-5}
{Alg-Rect1}  & 30.95 & 493 & 948.94 & $\alpha = 0.07 $ \\
{Alg-Rect2}   & 31.62 &  8  & 22.20 & $\alpha = 0.0003$\\
{Alg-Rect3}    & \textbf{31.61} & 7  & \textbf{13.50} & $\alpha = 0.0003$\\
\cline{1-5}
{Alg-BC4}    & 31.49  & 29    & 16.56 & $\alpha = 0.0018$  \\
{Alg-BC4--NS}   & 31.25 & 15 & 10.32 & $\mu = 6$   \\
\cline{1-5}
{FA-MD} & 30.87     &  & 90.85  & $\lambda =  0.001$  \\
{TV-MD} & 31.17     &  & 47.61  & $\lambda = 0.01$  \\
{FTVd} & 30.50      &  & 1.75 & $1/\alpha = 0.0013$\\
\end{tabular}
\end{small}
\caption{Example 1: PSNR, number of iterations, and CPU time in seconds for the best regularization parameter (maximum PSNR) reported in the last column. For our algorithms $\mu=10$ except for Alg-BC4--NS.}
\label{tab:Test1}
\end{center} \end{table}

\begin{figure}
\centering
 \subfigure[Alg-BC1]{
    \includegraphics[width=0.18\textwidth]{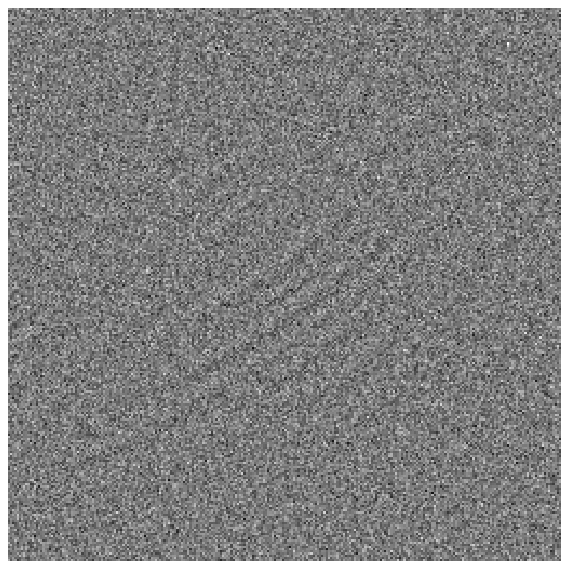}}
 \subfigure[Alg-Rect3]{
    \includegraphics[width=0.18\textwidth]{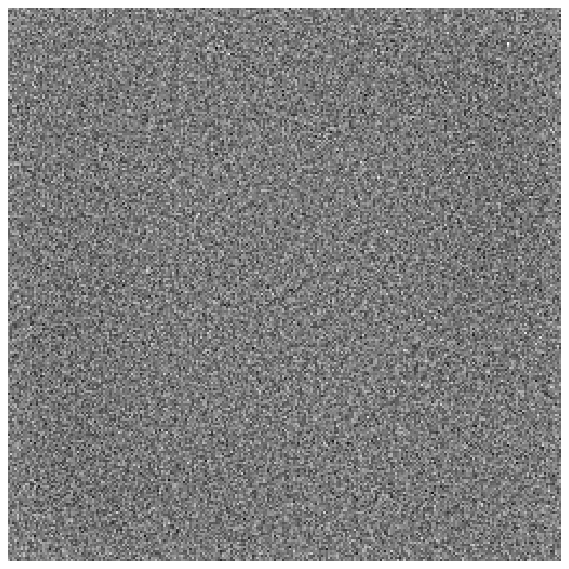}}
 \subfigure[Alg-BC4]{
    \includegraphics[width=0.18\textwidth]{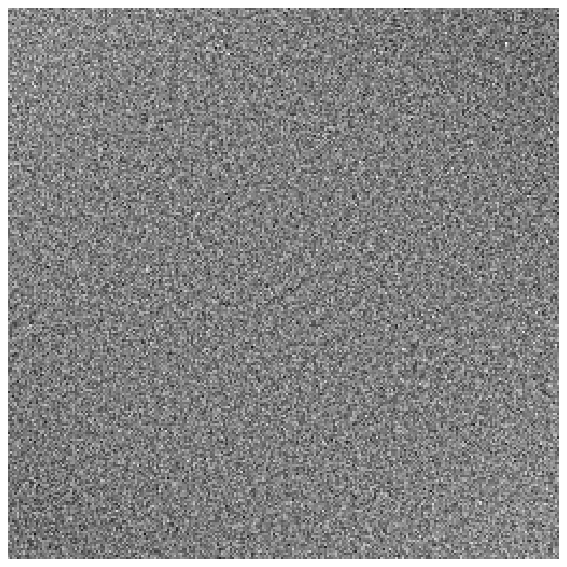}}
 \subfigure[FA-MD]{
    \includegraphics[width=0.18\textwidth]{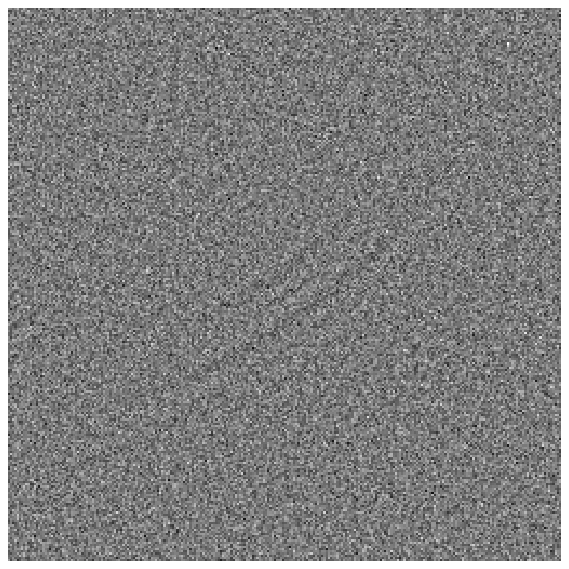}}
 \subfigure[TV-MD]{
    \includegraphics[width=0.18\textwidth]{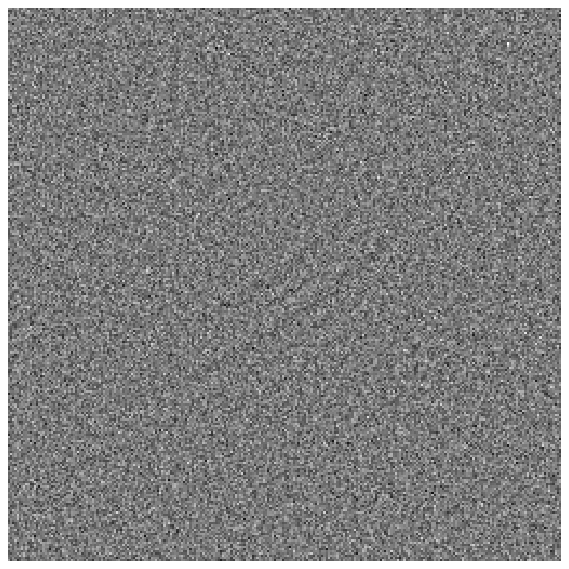}}
\caption{Example 1: residual image $\gb-A \tilde \fb$, where $ \tilde \fb$ is computed by different algorithms.}
\label{fig:Res1}
\end{figure}

\begin{figure}
\centering
 \subfigure[true image]{
    \includegraphics[width=0.175\textwidth]{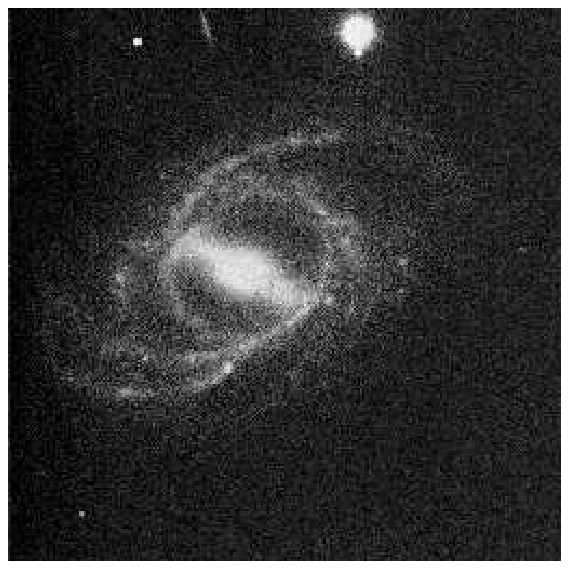}}
 \subfigure[PSF image]{
    \includegraphics[width=0.175\textwidth]{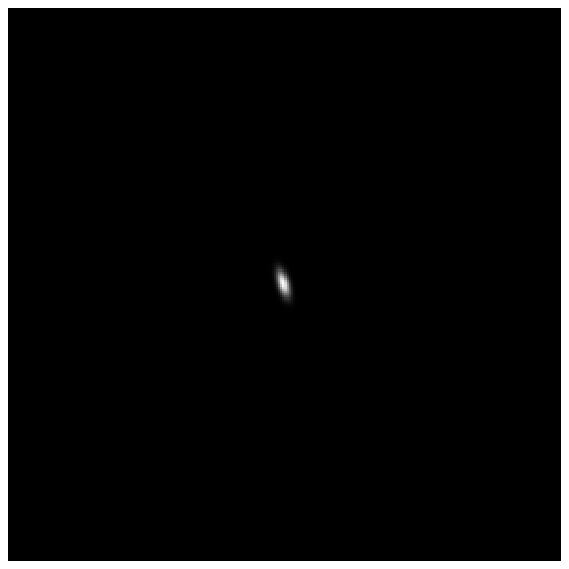}}
 \subfigure[observed image]{
   \includegraphics[width=0.175\textwidth]{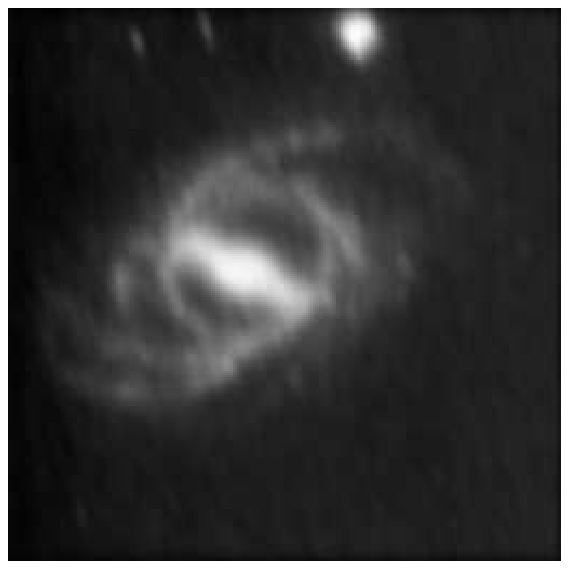}}
 \subfigure[Alg-BC1]{
    \includegraphics[width=0.175\textwidth]{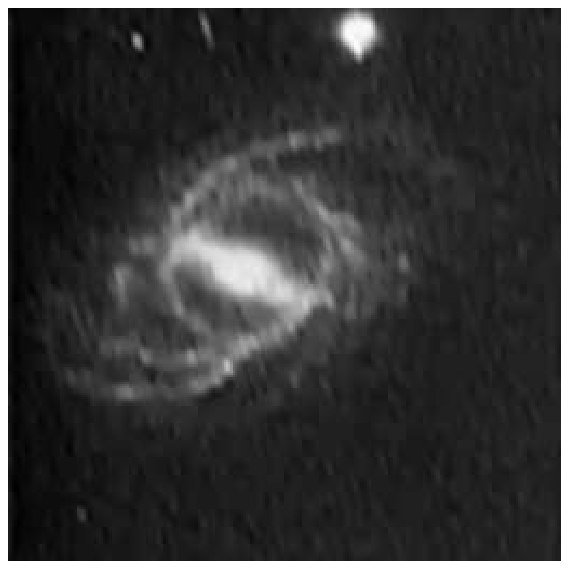}}
 \subfigure[Alg-BC3]{
   \includegraphics[width=0.175\textwidth]{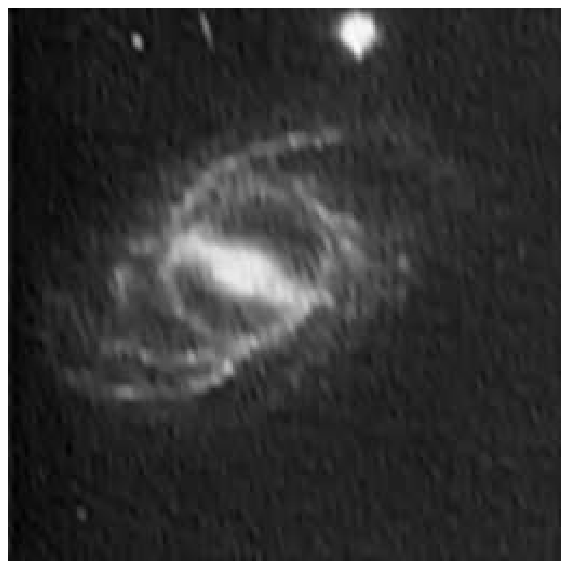}}
 \subfigure[Alg-Rect3]{
   \includegraphics[width=0.175\textwidth]{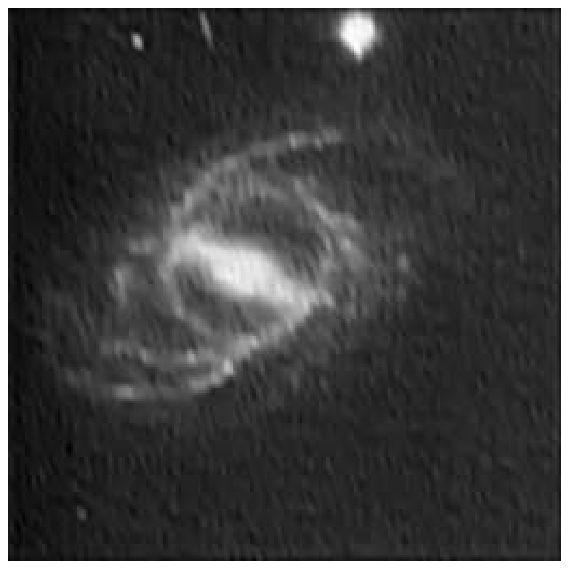}}
\subfigure[Alg-BC4]{
   \includegraphics[width=0.175\textwidth]{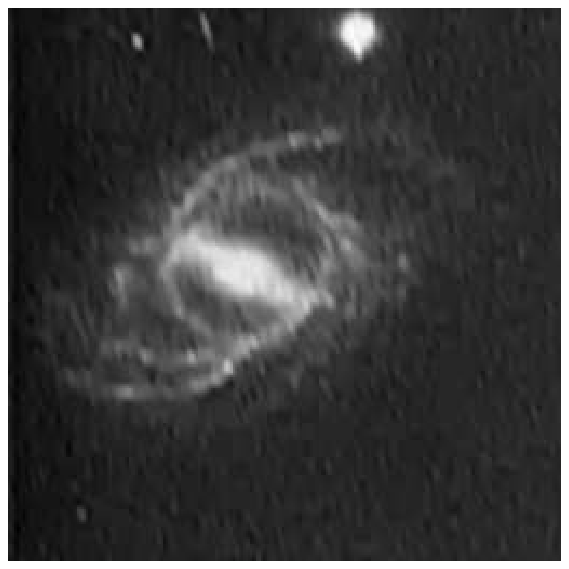}}
 \subfigure[Alg-BC4--NS]{
   \includegraphics[width=0.175\textwidth]{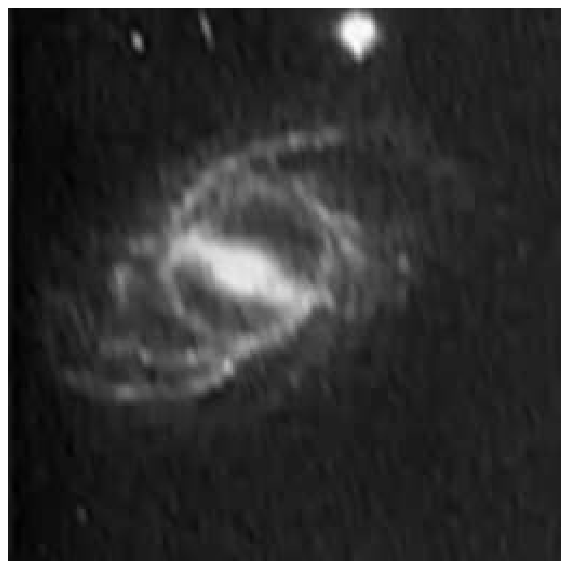}}
 \subfigure[FA-MD]{
    \includegraphics[width=0.175\textwidth]{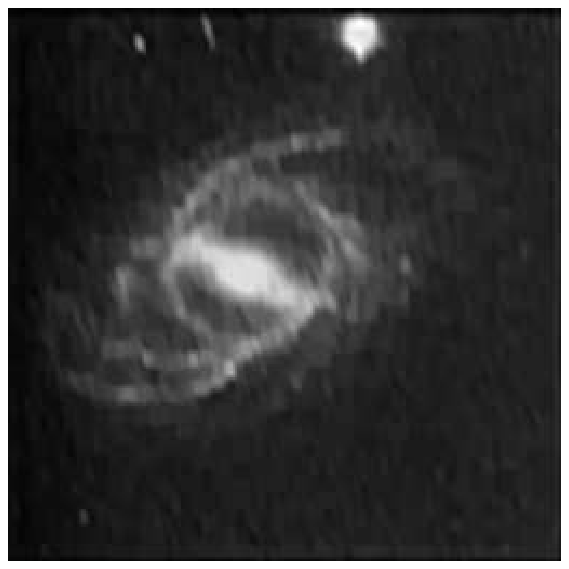}}
 \subfigure[FTVd]{
   \includegraphics[width=0.175\textwidth]{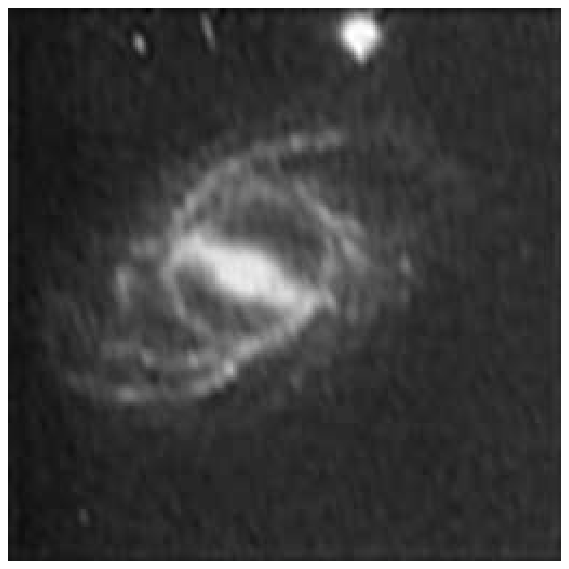}}
\caption{Example 2: true image, PSF, observed image, and restored images.}
\label{fig:Test2}
\end{figure}

\subsection{Example 2: Galaxy image }
We consider another astronomical example with the $256\times256$ image in Figure~\ref{fig:Test2}~(a)
corrupted by oblique Gaussian blur taken from the ``GaussianBlur422'' test problem in \cite{RestoreTools}, see Figure~\ref{fig:Test2}~(b). We add a $2\%$ of Gaussian white noise to obtain the observed image in Figure~\ref{fig:Test2}~(c).
We impose zero BCs and the computational properties of the different algorithms are the same as in Example~1.

Table \ref{tab:Test2} shows that Alg-BC4 is the best algorithm, since it obtains about the same PSNR of Algorithm~2, but with about 1/4 of the CPU time. The variant Alg-BC4--NS with a nonstationary choice of the preconditioner results to be very effective and comparable with the other algorithms based on the BC model avoiding the choice of parameter $\alpha$. 
Differently, the rectangular approach gives a slightly lower PSNR with a larger CPU time.
Concerning the other methods, the same observations reported for Example 1 still apply.
Figure \ref{fig:Test2} shows some of the corresponding restored images.

\begin{table}
\begin{center}\begin{small}
\begin{tabular}
{p{60pt}|p{25pt}p{20pt}p{60pt}|p{90pt}} \hline
Algorithm  & PSNR & Iter. & CPU time(s) & Regular. parameter  \\
\hline
\cline{1-5}
{Alg-BC1}  & 25.02  & 21   & 16.57  & $\alpha = 0.04$      \\
{Alg-BC2}   & 25.07  &  12  & 22.88 & $\alpha = 0.008$ \\
{Alg-BC3}  & 25.05  & 27    & 17.58 & $\alpha = 0.02$  \\
\cline{1-5}
{Alg-Rect1}  & 24.91 & 25   & 42.14 & $\alpha = 0.06 $ \\
{Alg-Rect2}   & 24.98 & 12   &  28.21& $\alpha = 0.007$\\
{Alg-Rect3}    & 24.96 & 21   & 36.18& $\alpha = 0.01$\\
\cline{1-5}
{Alg-BC4}    & \textbf{25.06}  & 12    & \textbf{6.21} & $\alpha = 0.008$  \\
{Alg-BC4--NS}   & 25.01  & 29   & 17.10  & $\mu = 4$   \\
\cline{1-5}
{FA-MD} & 24.50     &  &  77.05 & $\lambda = 0.02$  \\
{TV-MD} & 24.55     &  &   68.91 & $\lambda = 0.09$  \\
{FTVd} &  24.62  &  & 1.51 & $1/\alpha = 0.027$\\
\end{tabular}
\end{small}
\caption{Example 2: PSNR and CPU time for the best regularization parameter (maximum PSNR).
For our algorithms $\mu= 10$ except for Alg-BC4--NS.}\label{tab:Test2}
\end{center} \end{table}

\begin{figure}
\centering
 \subfigure[true image]{
    \includegraphics[width=0.18\textwidth]{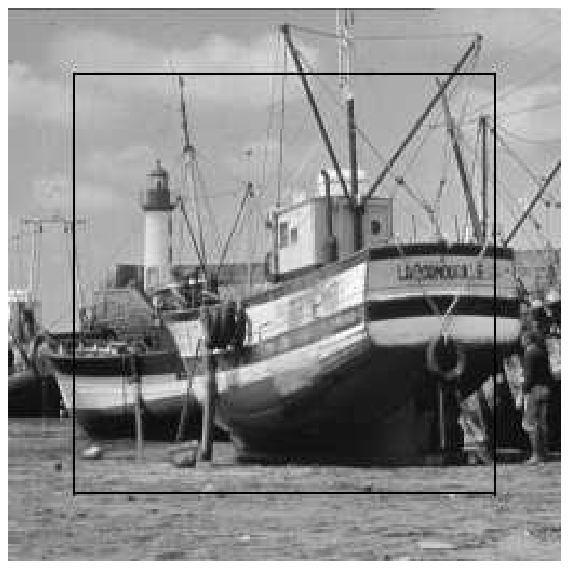}}
 \subfigure[PSF image]{
    \includegraphics[width=0.18\textwidth]{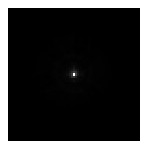}}
 \subfigure[observed image]{
    \includegraphics[width=0.18\textwidth]{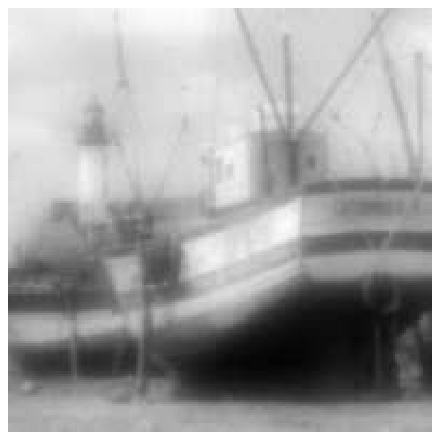}}
 \subfigure[Alg-BC1]{
    \includegraphics[width=0.18\textwidth]{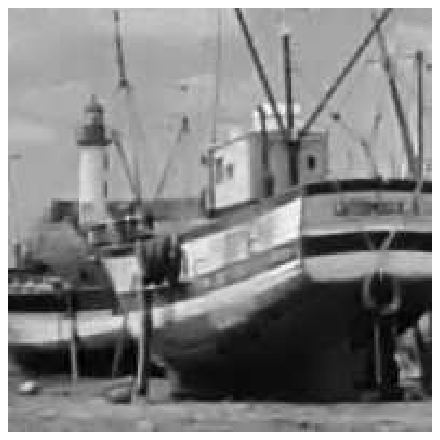}}
 \subfigure[Alg-BC2]{
    \includegraphics[width=0.18\textwidth]{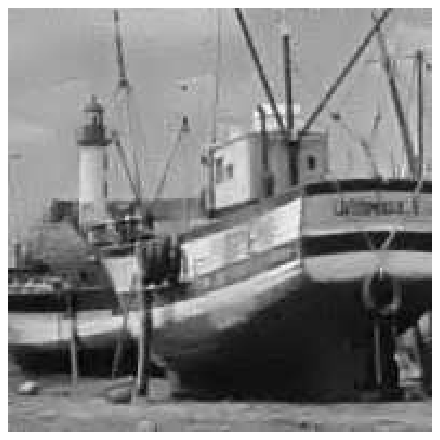}}
 \subfigure[Alg-Rect3]{
    \includegraphics[width=0.18\textwidth]{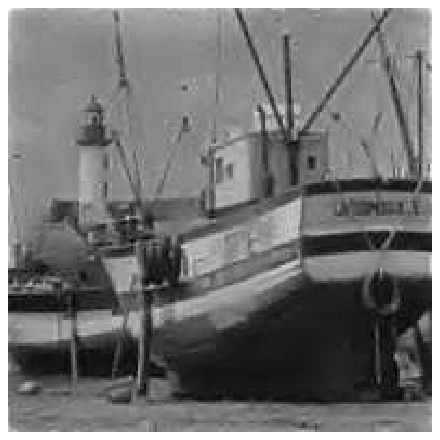}}
 \subfigure[Alg-BC4]{
    \includegraphics[width=0.18\textwidth]{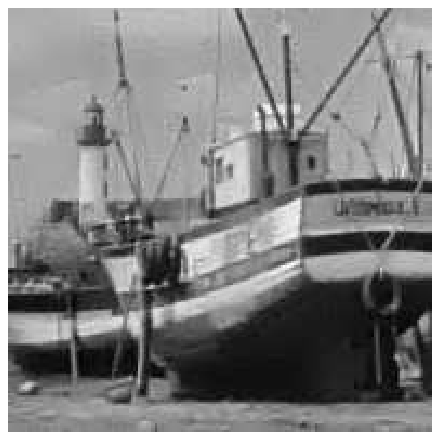}}
 \subfigure[Alg-BC4--NS]{
    \includegraphics[width=0.18\textwidth]{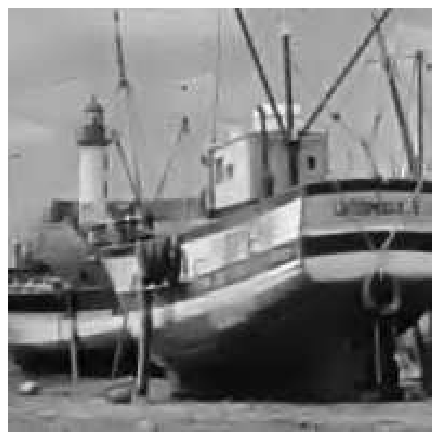}}
 \subfigure[FA-MD]{
    \includegraphics[width=0.18\textwidth]{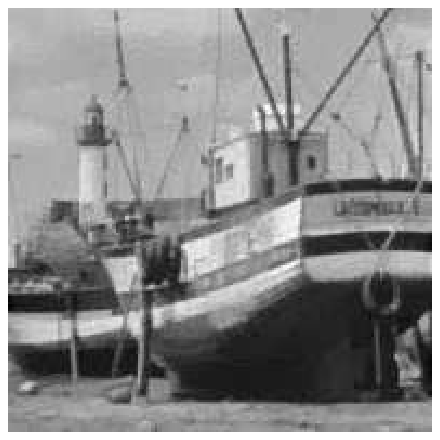}}
 \subfigure[TV-MD]{
    \includegraphics[width=0.18\textwidth]{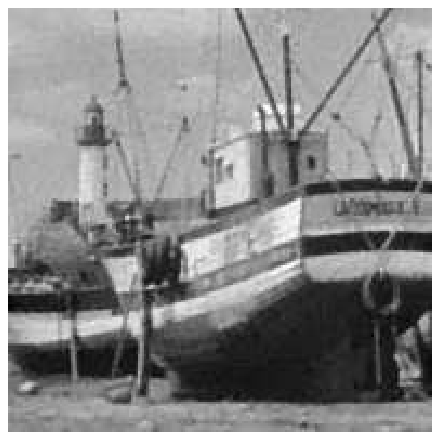}}
\caption{Example 3:  true image, PSF, observed image, and restored images.}
\label{fig:Test3}
\end{figure}

\subsection{Example 3: Boat image}

To set up a scenario of unknown boundaries, the observed image of size $196\times196$ is  obtained
convolving the full ($256\times256$) image using arbitrary BCs (periodic, for computational convenience) and then keeping only the pixels in the FOV $196\times196$ (i.e., those not depending on the BCs).
The FOV is denoted by a black box in the true image in Figure~\ref{fig:Test3} (a).
$1\%$ of white Gaussian noise is added to the $196\times196$ blurred image.

We impose antireflective BCs owing to the generic structure of that picture. 
Hence the preconditioner in Alg-BC3 is diagonalized by the antireflective transform, according to the structure of the matrix $A$. Alg-Rect3 uses the DCT as usual, while Alg-BC4 like Alg-BC1 and Alg-Rect1 use FFT since the PSF is not quadrantally symmetric.

Table~\ref{tab:Test3} shows the PSNR and the CPU time for the best restorations shown in Figure~\ref{fig:Test3}. Note that our algorithms with the rectangular matrix are less effective than the antireflective BC approach, leading to a lower PSNR.
To obtain reasonable restorations in the rectangular case we need a large $\mu$ and so we take a different $\mu$ for the two deblurring models (BC and rectangular matrix).
The best algorithm results to be Alg-BC4, since it combines a good restoration with a low CPU time. 

\begin{table}
\begin{center}\begin{small}
\begin{tabular}
{p{60pt}|p{25pt}p{20pt}p{60pt}|p{100pt}} \hline
Algorithm  & PSNR & Iter. & CPU time(s) & Regular. parameter  \\
\hline
\cline{1-5}
{Alg-BC1}  & 29.43  & 97    & 34.26  & $\mu=20, \quad \alpha=0.37$ \\
{Alg-BC2}   & 30.11  & 11    & 17.21 & $\mu=20, \quad \alpha = 0.025$\\
{Alg-BC3}     & 30.09  & 10   & 4.03 & $\mu=20, \quad \alpha=0.022$  \\
\cline{1-5}
{Alg-Rect1}    & 27.19  & 74  & 32.63 & $\mu=200, \quad \alpha=0.03$ \\
{Alg-Rect2} & 27.10  & 74   & 45.95  & $\mu=200, \quad \alpha=0.03$ \\
{Alg-Rect3}     & 27.22  & 74   &33.14& $\mu=200, \quad \alpha=0.03$ \\
\cline{1-5}
{Alg-BC4}  & \textbf{30.17}  & 13    & \textbf{3.67} & $\mu=20, \quad \alpha=0.03$  \\
{Alg-BC4--NS} & 29.77  & 60    & 19.57  & $\mu=30$ \\
\cline{1-5}
{FA-MD} & 29.61     & & 15.95 & $\lambda = 0.04$   \\
{TV-MD} & 29.87      && 16.74  & $\lambda = 0.1$ \\
{FTVd} & 28.95     && 0.73 & $1/\alpha = 0.0069$
\end{tabular}
\end{small}
\caption{Example 3:  PSNR and CPU time for the best regularization parameter (maximum PSNR).}
\label{tab:Test3}
\end{center} \end{table}

\begin{figure}
\centering
 \subfigure[true image]{
    \includegraphics[width=0.18\textwidth]{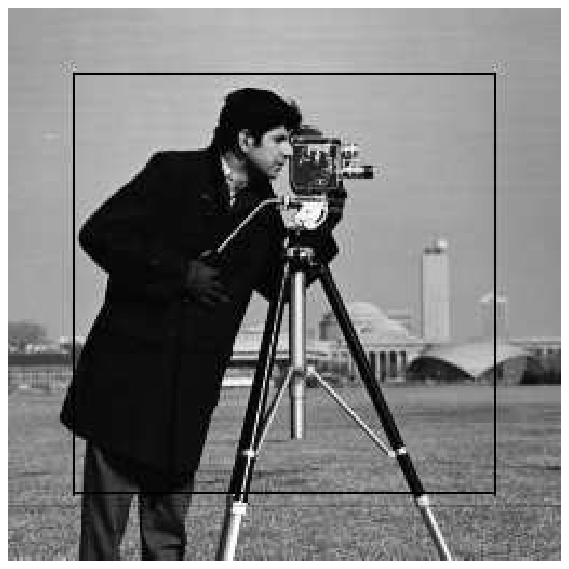}}
 \subfigure[PSF ($31 \times 31$)]{
    \includegraphics[width=0.18\textwidth]{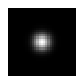}}
 \subfigure[observed image]{
    \includegraphics[width=0.18\textwidth]{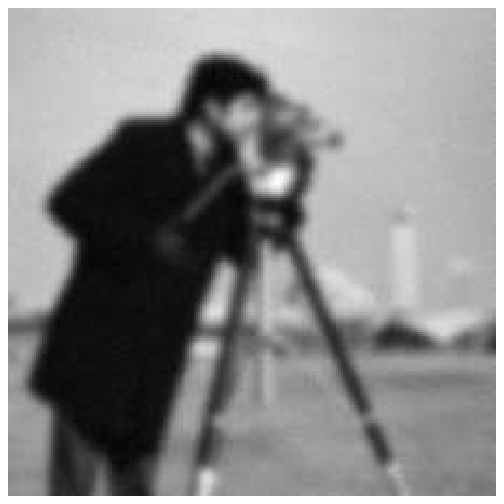}}
 \subfigure[Alg-BC1]{
    \includegraphics[width=0.18\textwidth]{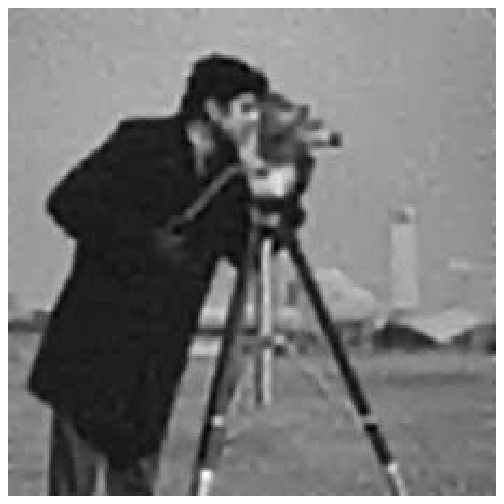}}
 \subfigure[Alg-BC3]{
    \includegraphics[width=0.18\textwidth]{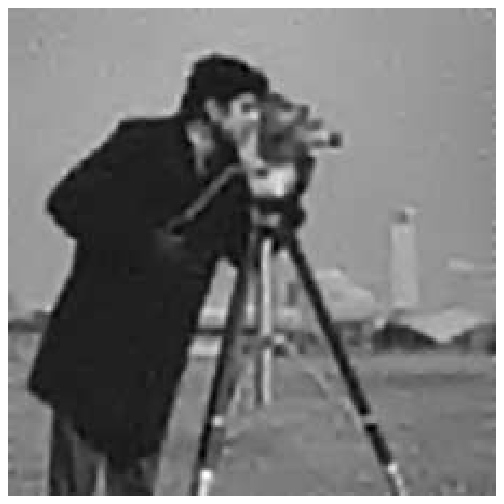}}
 \subfigure[Alg-Rect1]{
    \includegraphics[width=0.18\textwidth]{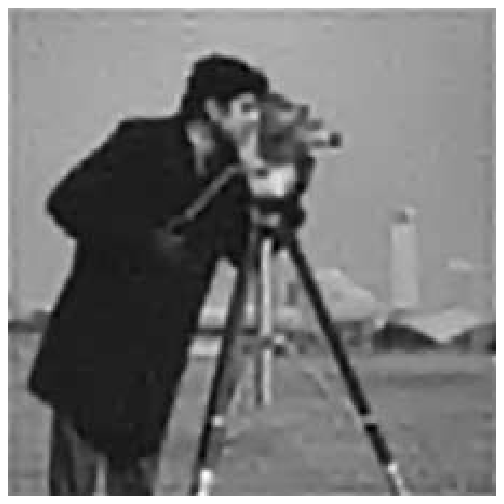}}
 \subfigure[Alg-BC4]{
    \includegraphics[width=0.18\textwidth]{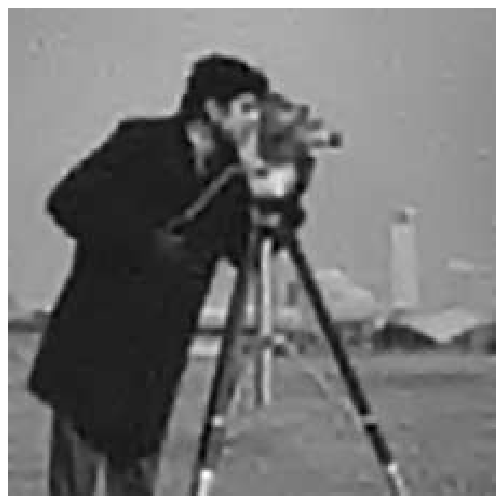}}
 \subfigure[Alg-BC4--NS]{
    \includegraphics[width=0.18\textwidth]{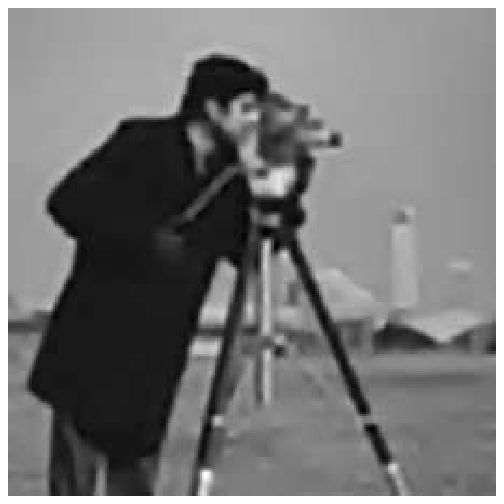}}
 \subfigure[TV-MD]{
    \includegraphics[width=0.18\textwidth]{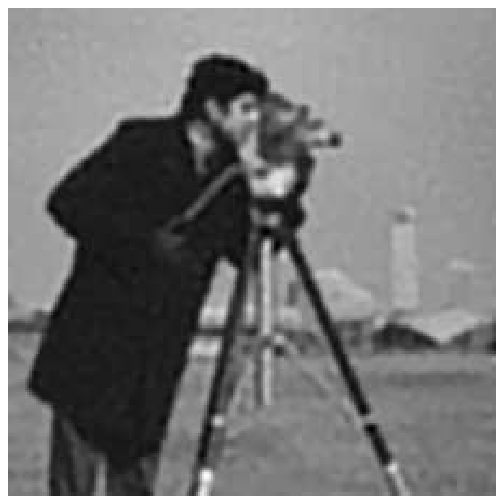}}
 \subfigure[FTVd]{
    \includegraphics[width=0.18\textwidth]{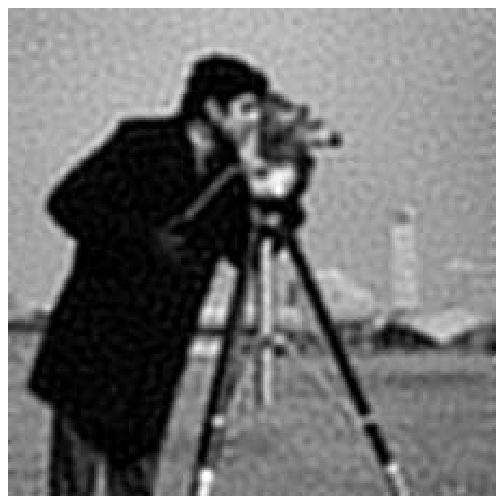}}
\caption{Example 4: true image, PSF, observed image, and restored images.}
\label{fig:Test4}
\end{figure}

\subsection{Example 4: Cameraman with Gaussian blur}
In this example we consider the classical Cameraman image
$256\times256$ distorted by a $31\times31$ Gaussian blur
with standard deviation 2.5 and a $2\%$ of white Gaussian noise (see Figure~\ref{fig:Test4}).
The size of the observed image is $226 \times 226$ according to the support of the PSF.

The PSF is quadrantally symmetric and hence, following our discussion in Section~\ref{sect:BCCBprec}, Algorithm 1 is implemented using the DCT instead of FFT. Clearly $\widetilde H = H$ and so  Alg-Rect3 reduces to Alg-Rect1 (they are really the same algorithm).
We impose antireflective BCs and hence Alg-BC3 is the standard MLBA \eqref{eq:MLBA2} with $P=(AA^T+\alpha I)^{-1}$, but recalling that we are using the reblurring approach where $A^T$ is replaced by $A'$. On the other hand Alg-BC2 is no longer useful since the matrix-vector product with the matrix $P=(AA^T+\alpha I)^{-1}$ can be computed by two antireflective transforms without requiring the PCG: in fact, the use of preconditioning would represent an unnecessary increase of the CPU time without increasing the PSNR with respect to Alg-BC3.
Finally, Alg-BC4 is implemented by DCT like Alg-BC1.

Table~\ref{tab:Test4} shows that all the compared methods in this example provide comparable results. Nevertheless, it is interesting to observe that Alg-BC4 gives a slightly better restoration with a lower CPU time than the standard MLBA, i.e., Alg-BC3. TV-MD computes again a comparable restoration, but with more than a double CPU time.
Figure~\ref{fig:Test4} shows the restored images.

\begin{table}
\begin{center}\begin{small}
\begin{tabular}
{p{60pt}|p{25pt}p{20pt}p{60pt}|p{90pt}} \hline
Algorithm  & PSNR & Iter. & CPU time(s) & Regular. parameter  \\
\hline
\cline{1-5}
{Alg-BC1}  & 23.67  & 51    &  20.19 & $\alpha = 0.05$      \\
{Alg-BC3}   & 23.74  & 17   & 6.03 & $\alpha = 0.01$ \\
\cline{1-5}
{Alg-Rect1}  & 23.59 & 24  & 12.44 & $\alpha = 0.02 $ \\
{Alg-Rect2}   & 23.64 & 17   & 16.49 & $\alpha = 0.009$\\
\cline{1-5}
{Alg-BC4}    & \textbf{23.76}  & 17    &  \textbf{5.60} & $\alpha = 0.01$  \\
{Alg-BC4--NS}   & 23.53  & 39    & 14.78 & $\mu = 40$   \\
\cline{1-5}
{FA-MD} & 23.44     &  & 14.10  & $\lambda = 0.01$ \\
{TV-MD} & 23.55    &  & 13.37  & $\lambda = 0.11$  \\
{FTVd} &  23.10     &  & 0.89 & $1/\alpha = 0.0088$\\
\end{tabular}
\end{small}
\caption{Example 4:  PSNR and CPU time for the best regularization parameter (maximum PSNR). For our algorithms $\mu=40$.}
\label{tab:Test4}
\end{center} \end{table}

To test the ability of the different algorithms in dealing with the boundary effects, Figure~\ref{fig:Res4}
shows the North-East corner of the residual images.
We can see that Alg-BC1 and FTVd have some ringing effects at the boundary, while Algorithm~4 and TV-MD do not show any particular distortion at the boundary.

\begin{figure}
\centering
 \subfigure[Alg-BC1]{
    \includegraphics[width=0.18\textwidth,clip=true,trim=3cm 3cm 0cm 0cm]{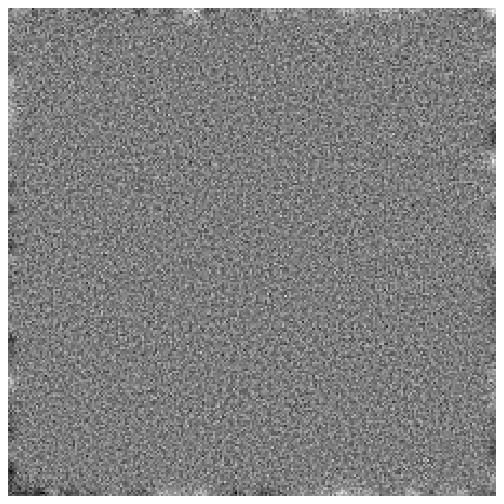}}
 \subfigure[Alg-BC4]{
    \includegraphics[width=0.18\textwidth,clip=true,trim=3cm 3cm 0cm 0cm]{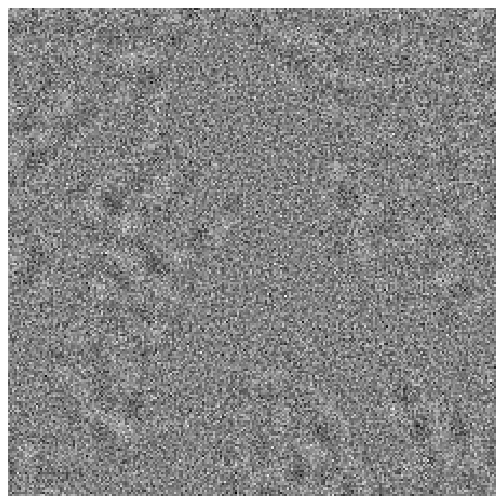}}
 \subfigure[Alg-BC4--NS]{
    \includegraphics[width=0.18\textwidth,clip=true,trim=3cm 3cm 0cm 0cm]{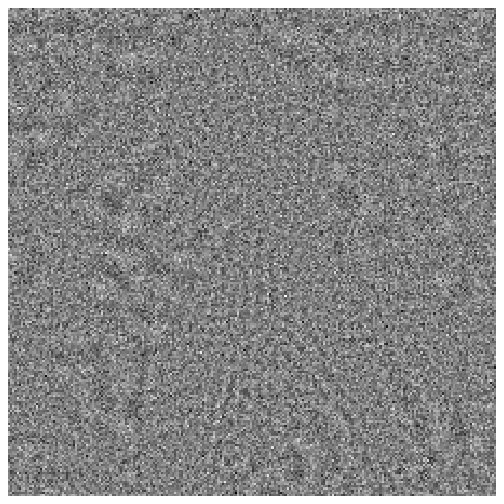}}
 \subfigure[TV-MD]{
    \includegraphics[width=0.18\textwidth,clip=true,trim=3cm 3cm 0cm 0cm]{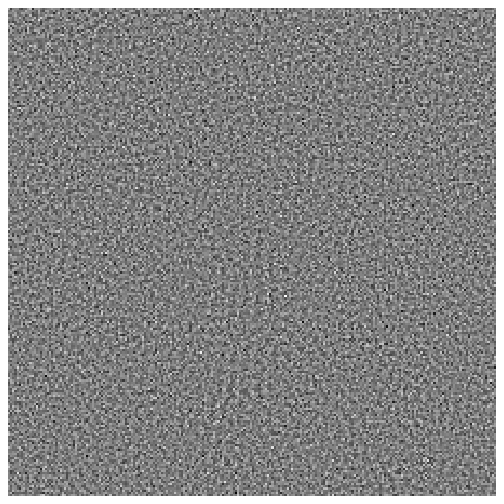}}
 \subfigure[FTVd]{
    \includegraphics[width=0.18\textwidth,clip=true,trim=3cm 3cm 0cm 0cm]{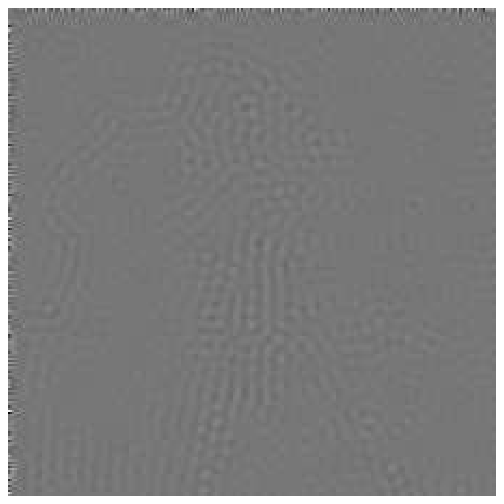}}
\caption{Example 4: Nort-East corner of the residual images.}
\label{fig:Res4}
\end{figure}

\section{Conclusion}\label{sect:concl}
In this paper, we have investigated several regularization preconditioning strategies for the MLBA applied to the synthesis approach with accurate restoration models for image
deblurring and unknown boundaries. Our numerical results shows that Alg-BC4, which combines the favorite BCs (depending on the problem) with an approximated Tikhonov regularization preconditioner, represents a robust and effective algorithm. Indeed, it provides accurate restorations in all our examples with a reduced CPU time also in comparison to the state of the art algorithms \cite{AF13,BCDS14} and the standard MLBA when available, cf. Example 4.

We have investigated only the synthesis approach, but the same preconditioning strategies can be applied to the analysis and the balanced approach \cite{COS2,STY} as well.
Moreover, possible future investigations could consider the use of a preconditioner obtained by a small rank approximation of the PSF as in \cite{KN}, 
strategies to estimate the parameter $\alpha$, and nonstationary sequences to approximate the best $\alpha$ avoiding its estimation like in \cite{HDC14}.

\subsection*{Acknowledgement}
We would like to thank the Referees and the Editor for their valuable comments and suggestions,
which helped us to improve the readability and the content of the paper. We also thank Stefano Serra-Capizzano for the useful discussions.

The work of the first and fourth author is supported in part by 973 Program (2013CB329404), NSFC (61370147,61170311), Sichuan Province Sci. \& Tech. Research Project (2012GZX0080).
The work of the second and third author is partly  supported by PRIN 2012 N. 2012MTE38N.


\end{document}